\documentclass{elsarticle}

\usepackage{amssymb}
\usepackage{amsmath}
\usepackage{bm,bbm}
\usepackage{comment}
\usepackage{epsfig}
\usepackage{float}
\usepackage{graphics}
\usepackage{graphicx}
\usepackage{multicol}
\usepackage{multirow}
\usepackage{soul}       
\usepackage{subfigure}

\usepackage[margin=2.5cm]{geometry}

\biboptions{sort&compress}

\newtheorem{theorem}{Theorem}[section]
\newtheorem{lemma}[theorem]{Lemma}
\newtheorem{corollary}[theorem]{Corollary}

\newtheorem{remark}[theorem]{Remark}

\newenvironment{proof}{\textit{Proof.}}

\usepackage{color}
\definecolor{lightblue}{rgb}{0.22,0.45,0.70}

\definecolor{MyDarkGreen}{rgb}{0,0.45,0}

\newcommand{\RED} [1]{#1}
\newcommand{\BLUE}[1]{#1}

\def\trait #1 #2 #3 {\vrule width #1pt height #2pt depth #3pt}
\def\fin{\hfill
        \trait .3 5 0
        \trait 5 .3 0
        \kern-5pt
        \trait 5 5 -4.7
        \trait 0.3 5 0
\medskip}
\newcommand{\ENDPROOF}{\fin}

\newcommand{\ASSUM}[2]{(\textsf{#1#2})}  
\newcommand{\TERM} [2]{\textsf{#1}_{#2}}  
\newcommand{\DOFS} [2]{\textbf{dofs}_{#2}\big({#1}\big)}

\newcommand{\dx}{\,d\xs}
\newcommand{\dy}{\,d\ys}

\newcommand{\dt}{\,d\ts}

\newcommand{\KER} {\textrm{ker}}
\newcommand{\RANK}{\textrm{rank}}
\newcommand{\REAL}{\mathbbm{R}}

\newcommand{\restrict}[2]{{#1}_{|{#2}}}

\newcommand{\PGRAPH}[1]{\medskip\noindent\textbf{#1}.}
\newcommand{\EOD}{\end{document}}
\newcommand{\EOPD}{\end{proof}\end{document}}

\newcommand{\esssup}{\textrm{ess~sup}}

\definecolor{MyDarkGreen}{rgb}{0,0.45,0}

\newcommand{\bv}{\mathbf{b}}

\newcommand{\ev}{\mathbf{e}}
\newcommand{\fv}{\mathbf{f}}

\newcommand{\vv}{\mathbf{v}}

\newcommand{\xv}{\mathbf{x}}
\newcommand{\yv}{\mathbf{y}}

\newcommand{\xvP}{\xv_{\P}}

\newcommand{\xvE}{\xv_{\E}}

\newcommand{\bvh}{\bv_{\hh}}
\newcommand{\fvh}{\fv_{\hh}}

\newcommand{\as}{a}
\newcommand{\bs}{b}

\newcommand{\ds}{d}
\newcommand{\es}{e}
\newcommand{\fs}{f}

\newcommand{\hs}{h}

\newcommand{\ks}{k}

\newcommand{\ms}{m}
\newcommand{\ns}{n}

\newcommand{\ps}{p}
\newcommand{\qs}{q}
\newcommand{\rs}{r}

\newcommand{\ts}{t}
\newcommand{\us}{u}
\newcommand{\vs}{v}
\newcommand{\ws}{w}
\newcommand{\xs}{x}
\newcommand{\ys}{y}
\newcommand{\zs}{z}

\newcommand{\Cs}{C}
\newcommand{\Ds}{D}
\newcommand{\Es}{E}
\newcommand{\Fs}{F}

\newcommand{\Is}{I}
\newcommand{\Js}{J}

\newcommand{\Ms}{M}
\newcommand{\Ns}{N}

\newcommand{\Ps}{P}

\newcommand{\Ss}{S}
\newcommand{\Ts}{T}
\newcommand{\Us}{U}
\newcommand{\Vs}{V}

\newcommand{\Xs}{X}

\newcommand{\sPa}{S^{\P}_{a}}
\newcommand{\sPm}{S^{\P}_{m}}

\newcommand{\xsP}{\xs_{\P}}
\newcommand{\ysP}{\ys_{\P}}

\newcommand{\xsE}{\xs_{\E}}
\newcommand{\ysE}{\ys_{\E}}

\newcommand{\FsP}[1]{\Fs^{\P}_{#1}}
\newcommand{\Css}{\Cs^*}

\newcommand{\vsp} {\vs_1}
\newcommand{\vspp}{\vs_2}

\newcommand{\matD}{\mathsf{D}}

\newcommand{\matH}{\mathsf{H}}

\newcommand{\matM}{\mathsf{M}}

\newcommand{\calA}{\mathcal{A}}
\newcommand{\calB}{\mathcal{B}}

\newcommand{\calE}{\mathcal{E}}

\newcommand{\calG}{\mathcal{G}}
\newcommand{\calH}{\mathcal{H}}
\newcommand{\calI}{\mathcal{I}}

\newcommand{\calK}{\mathcal{K}}

\newcommand{\calM}{\mathcal{M}}

\newcommand{\calO}{\mathcal{O}}
\newcommand{\calP}{\mathcal{P}}

\newcommand{\calT}{\mathcal{T}}

\newcommand{\Ash}{\calA_{\hh}}

\newcommand{\asP} {\as^{\P}}

\newcommand{\PinP}[1]{\Pi^{\nabla,\P}_{#1}}
\newcommand{\Pin} [1]{\Pi^{\nabla}_{#1}}

\newcommand{\PizP}[1]{\Pi^{0,\P}_{#1}}
\newcommand{\Piz} [1]{\Pi^{0}_{#1}}

\newcommand{\PiLSP}[1]{\TILDE{\Pi}^{\P}_{#1}}
\newcommand{\PiLSsP}[1]{\TILDE{\Pi}^{\P,*}_{#1}}
\newcommand{\PiLS} [1]{\TILDE{\Pi}_{#1}}

\newcommand{\matPiLSP}{\mathbf{\widetilde{\Pi}}^{\P,\ast}}

\newcommand{\WS}[1] {W^{#1}}
\newcommand{\LS}[1] {L^{#1}}
\newcommand{\HS}[1] {H^{#1}}

\newcommand{\CS}[1] {C^{#1}}

\newcommand{\HSzr}[1]{H^{#1}_{0}}

\newcommand{\PS}[1] {\mathbbm{P}_{#1}}

\newcommand{\qscalP} [2]{[#1,#2]_{\P}}
\newcommand{\qScalP} [2]{\big[#1,#2\big]_{\P}}
\newcommand{\qSCALP} [2]{\left[#1,#2\right]_{\P}}

\newcommand{\scal}  [2]{(#1,#2)}
\newcommand{\scalP} [2]{(#1,#2)_{\P}}
\newcommand{\scalX} [2]{(#1,#2)_{\Xs}}

\newcommand{\Scal}  [2]{\left(#1,#2\right)}
\newcommand{\ScalP} [2]{\left(#1,#2\right)_{\P}}

\newcommand{\abs}   [1]{|#1|}
\newcommand{\Abs}   [1]{\big|#1\big|}
\newcommand{\ABS}   [1]{\left|#1\right|}

\newcommand{\snorm} [2]{|#1|_{#2}}

\newcommand{\SNORM} [2]{\left|#1\right|_{#2}}

\newcommand{\norm}  [2]{\|#1\|_{#2}}
\newcommand{\Norm}  [2]{\big\|#1\big\|_{#2}}
\newcommand{\NORM}  [2]{\left\|#1\right\|_{#2}}

\newcommand{\TNORM} [2]{\left|\!\left|\!\left|#1\right|\!\right|\!\right|_{#2}}
\newcommand{\Tnorm} [2]{\big|\!\big|\!\big|#1\big|\!\big|\!\big|_{#2}}

\renewcommand{\P} {\textsf{E}}            
\newcommand  {\E} {\textsf{e}}            

\newcommand{\mP}{\ABS{\P}}

\newcommand{\mE}{\ABS{\E}}

\newcommand{\hP}{\hh_{\P}}

\newcommand{\hE}{\hh_{\E}}

\newcommand{\hh}{h}
\newcommand{\Th}{\Omega_{\hh}} 

\newcommand{\NT}{N}

\newcommand{\INTP}{\footnotesize{I}}
\newcommand{\TILDE}[1]{\widetilde{#1}}

\newcommand{\bsh}{\bs_{\hh}}
\newcommand{\zsh}{\zs_{\hh}}

\newcommand{\ush} {\us_{\hh}}

\newcommand{\vsh} {\vs_{\hh}}
\newcommand{\vsI} {\vs_{\INTP}}

\newcommand{\fsh} {\fs_{\hh}}

\newcommand{\wsh} {\ws_{\hh}}

\newcommand{\Ush} {\Us_{\hh}}

\newcommand{\Vsh} {\Vs_{\hh}}

\newcommand{\Vsht}{\TILDE{\Vs}_{\hh}}

\newcommand{\Ssh} {\Ss_{\hh}}

\newcommand{\vvh} {\vv_{\hh}}

\newcommand{\ash} {\as_{\hh}}
\newcommand{\msh} {\ms_{\hh}}

\newcommand{\ashP}{\as_{\hh}^{\P}}
\newcommand{\mshP}{\ms_{\hh}^{\P}}

\newcommand{\Ih}{\Is_{\hh}}

\newcommand{\half}{1\slash{2}}
\newcommand{\calKh}{\calK_{\hh}}
\newcommand{\calMh}{\calM_{\hh}}
\newcommand{\calMhP}{\calM_{\hh}^{\P}}

\newcommand{\uss}[1]{\us^{#1}}
\newcommand{\wss}[1]{\ws^{#1}}
\newcommand{\qss}[1]{\qs^{#1}}

\newcommand{\Ussh}[1]{\Ush^{#1}}

\newcommand{\ussI}[1]{\Ih\us^{#1}}

\newcommand{\ess} [1]{\es^{#1}}
\newcommand{\essh}[1]{\es_{\hh}^{#1}}
\newcommand{\etas}[1]{\eta^{#1}}
\newcommand{\thes}[1]{\theta^{#1}}

\newcommand{\fssh}[1]{\fsh^{#1}}                
\newcommand{\fss}[1]{\fs^{#1}}

\newcommand{\zetav}{{\bm\zeta}}

\newcommand{\NVP}{N^{\P}}

\begin{document}

\begin{frontmatter}

  \title{Virtual Element Approximation of Two-Dimensional\\ Parabolic
    Variational Inequalities}

  \author[IIT,chile]  {D.~Adak}
  \author[IMATI]{G.~Manzini}
  \author[IIT]  {S.~Natarajan}

  \address[IIT]{Department of Mechanical Engineering, Indian Institute
    of Technology Madras, Chennai-600036, India.}
    
    \address[chile]{GIMNAP, Departamento de Matem\'atica, Universidad
del B\'io-B\'io, Concepci\'on, Chile.}

  \address[IMATI]{IMATI, Consiglio Nazionale delle
    Ricerche, via Ferrata 1, 27100 Pavia, Italy }

  \begin{abstract}
    We design a virtual element method for the numerical treatment of
    the two-dimensional parabolic variational inequality problem on
    unstructured polygonal meshes.
    Due to the expected low regularity of the exact solution, the
    virtual element method is based on the lowest-order virtual
    element space that contains the subspace of the linear polynomials
    defined on each element.
    The connection between the nonnegativity of the virtual element
    functions and the nonnegativity of the degrees of freedom, i.e.,
    the values at the mesh vertices, is established by applying the
    Maximum and Minimum Principle Theorem.
    The mass matrix is computed through an approximate $\LS{2}$
    polynomial projection, whose properties are carefully investigated
    in the paper.
    We prove the well-posedness of the resulting scheme in two
    different ways that reveal the contractive nature of the VEM and
    its connection with the minimization of quadratic functionals.
    The convergence analysis requires the existence of a nonnegative
    quasi-interpolation operator, whose construction is also discussed
    in the paper.
    The variational crime introduced by the virtual element setting
    produces five error terms that we control by estimating a suitable
    upper bound.
    Numerical experiments confirm the theoretical convergence rate for
    the refinement in space and time on three different mesh families including 
    distorted squares, nonconvex elements, and Voronoi tesselations.
  \end{abstract}
  

\end{frontmatter}



\section{Introduction}
\label{sec:Introduction}
Variational inequalities have been an active research field in the
last decades and has found many important applications in finance and
engineering~\cite{Kinderlehrer-Stampacchia:1980,Baiocchi-Capelo:1984,Rodrigues:1987-book}.
For example, they are used in the formulation of the one-phase Stefan
problem~\cite{Rodrigues:1987}.
The Allen—Cahn equation, one of the models of the kinetics of grain
growth in polycrystals, can be treated as a parabolic variational
inequality~\cite{Blowey-Elliott:1993}.
The American put option problem~\cite{Kazufumi-Kunisch:2006} becomes a
one-phase Stefan problem after a suitable change of
variable~\cite{Karatzas-Shreve:1998}.
The electrochemical machine problem is also modeled using variational
inequalities~\cite{Elliott:1980}.
Static contact problems, frictional contact problems, and thermal
expansion problems can be described using variational inequalities,
cf.~\cite{Capatina:2014}.
The numerical approximation of the solution to variational
inequalities has also been a challenging area of research since both
the design of numerical methods and the convergence analysis are not
straightforward~\cite{Glowinski-Lions-Tremolieres:1981}.

The Galerkin approach and, in particular, the finite element method
(FEM) has proven to be quite effective to this purpose.
The linear Galerkin FEM for the time-dependent parabolic variational
inequality (with zero obstacle) was originally proposed
in~\cite{Johnson:1976}.
In this paper, which is the most pertinent to our current work, a
priori error estimates in the $\LS{\infty}$ norm are derived assuming
that the solution is in $\LS{\infty}\big(0,T;\WS{2,p}(\Omega)\big)$
and its first derivative in time is in
$\LS{\infty}\big(0,T;\HSzr{1}(\Omega)\big)\cap\LS{\infty}\big(0,T;\LS{\infty}(\Omega)\big)$
(we explain this notation and provide a formal definition of these
functional spaces later in this section).
A priori estimates in the $\LS{2}$ norm are also derived for the
Galerkin method in~\cite{Fetter:1987} assuming that the solution is in
$\LS{2}(0,T;\LS{2}(\Omega))$ and under certain regularity assumptions
on the angles of each element of the triangulations.
In~\cite{Vuik:1990}, $\LS{2}$ error estimates are derived for a fully
discrete scheme based on the $\theta$-method in time.
Inspired by the American put option problem, \emph{a posteriori} error
estimates are studied
in~\cite{Moon-Nochetto-VonPedersdorrfs-Zhang:2007}.
In \cite{BencheikhLeHocine-Boulaaras-Haiour:2016}, the Authors derive
error estimates for the parabolic variational inequality problem in
the uniform norm.
Moreover, in
\cite{Berger:1977,Gudi:2019,Allegretto-Lin-Yang:2001,Sharma-Pani-Sharma:2018},
mathematical models of the parabolic obstacle problem related to the
American put option problem and the Stefan problem are investigated.

In this work, we consider the approach that was originally proposed
in~\cite{Johnson:1976} for solving a parabolic variational inequality
on triangular meshes and study how to generalize it to polygonal
meshes using the virtual element method
(VEM)~\cite{BeiraodaVeiga-Brezzi-Cangiani-Manzini-Marini-Russo:2013,Ahmad-Alsaedi-Brezzi-Marini-Russo:2013}.
Designing Galerkin schemes for meshes with polygonal elements in 2D
and polyhedral elements in 3D has been a major topic in the numerical
literature of partial differential equations of the last two decades.
Several classes of numerical methods have been designed that are
suitable to meshes with elements having very general geometric shapes.
Other than the VEM, a surely nonexhaustive list includes
the polygonal/polyhedral finite element method
(PFEM)~\cite{Sukumar-Malsch:2006,Sze-Sheng:2005,Bishop:2014},
the mimetic finite difference method (MFD)~\cite{%
  Lipnikov-Manzini-Shashkov:2014,%
  BeiraodaVeiga-Lipnikov-Manzini:2014,%
  BeiraodaVeiga-Lipnikov-Manzini:2011},
the hybridizable discontinuous Galerkin (HDG) method and the hybrid
high-order (HHO) method
\cite{Cockburn-DiPietro-Ern:2016,DiPietro-Ern:2015,DiPietro-Droniou:2019,DiPietro-Droniou-Manzini:2018}.
Pertinent to the topic of our work are also the papers of
References~\cite{Antonietti-BeiraodaVeiga-Verani:2013a,Antonietti-BeiraodaVeiga-Verani:2013b}.

The virtual element method was proposed as a variational reformulation
of the mimetic finite difference method of
References~\cite{Brezzi-Buffa-Lipnikov:2009,BeiraodaVeiga-Lipnikov-Manzini:2011}
for the Poisson equation,
and later extended to the numerical approximation of
general elliptic equations \cite{BeiraodaVeiga-Brezzi-Marini-Russo:2016-M3AS},
elasticity problems~\cite{Mora-Rivera:2020},
eigenvalue problems~\cite{%
  Mora-Rivera-Rodriguez:2015,%
  Mora-Rivera-Rodriguez:2017,%
  Certik-Gardini-Manzini-Vacca:2018,%
  Gardini-Vacca:2018},
Stokes and Navier-Stokes equations
\cite{Antonietti-BeiraodaVeiga-Mora-Verani:2014,Caceres-Gatica:2017,Gatica-Munar-Sequeira:2018,BeiradaVeiga-Lovadina-Vacca:2018,BeiradaVeiga-Vacca:2019},
and the Cahn-Hilliard equations
\cite{Antonietti-BeiraodaVeiga-Scacchi-Verani:2016}.
Furthermore,
the mixed formulation~\cite{Brezzi-Falk-Marini:2014}
the nonconforming formulation
\cite{AyusodeDios-Lipnikov-Manzini:2016,Cangiani-Gyrya-Manzini:2016,Cangiani-Manzini-Sutton:2016},
and the enriched
formulation~\cite{Benvenuti-Chiozzi-Manzini-Sukumar:2019}
have been proposed and
a posteriori error estimations
\cite{%
  BeiraodaVeiga-Manzini:2015,%
  Cangiani-Georgoulis-Pryer-Sutton:2018,%
  BeiraodaVeiga-Manzini-Mascotto:2018,%
  Deng-Wang-Wei:2020}
have been derived for mesh adaptivity.
VEM for anisotropic polygonal discretizations are also found
in~\cite{Antonietti-Berrone-Verani-Weisser:2019}.

The VEM satisfies a Galerkin-type orthogonalization property on
polynomial subspaces and can be seen as a generalization of the FEM on arbitrary
polytopal meshes.
The finite dimensional approximation spaces consist of polynomial and
nonpolynomial functions that satisfy a partial differential equation
locally defined on the mesh elements.
The nonpolynomial functions are not known inside the elements, but the
degrees of freedom of the virtual element functions are carefully
chosen so that some polynomial projection operators are computable.
These projection operators make it possible to design computable
bilinear forms for the discrete variational formulation.
Since an explicit knowledge of the virtual element functions is not
required in the practical implementation, such ``virtual'' setting
works for very general shaped polytopal elements.
For example, nonconvex elements and elements with hanging nodes are
admissible and the latter do not require any special treatment.


Due to the expected low regularity of the solution, our method is
based on the lowest-order approximation space proposed
in~\cite{BeiraodaVeiga-Brezzi-Cangiani-Manzini-Marini-Russo:2013,Ahmad-Alsaedi-Brezzi-Marini-Russo:2013}.
The degrees of freedom are the vertex values and our VEM coincides
with the FEM of Reference~\cite{Johnson:1976} on all triangular
meshes.
The generalization to the virtual element framework is nontrivial and
the design of an effective VEM and its analysis is challenging for
several reasons that we illustrate below.
First, the variational formulation is given on the subset of the
nonnegative virtual element functions, which we identify with those
functions of the virtual element space  whose degrees of
freedom, i.e., the vertex values, are nonnegative.
The property that a function with nonnegative vertex values is
nonnegative is obvious for a linear polynomial interpolating such
values on a triangular element.
However, to prove that such property holds for a virtual element
function on a polygonal element is a nontrivial task.
In fact, such functions are not generally known in closed form, but
only as the solutions of an elliptic partial differential equation
that is locally set on the polygonal element.
We address this issue by noting that the lowest-order virtual element
space consists of functions that are harmonic inside each element and
have a continuous piecewise linear trace on the elemental boundary
given by the interpolation of the vertex values.
Consequently, we can prove the nonnegativity property by invoking the
Maximum and Minimum Principle Theorem~\cite{Gilbarg-Trudinger:2001}.
According to this theorem, a nonconstant harmonic function on a
compact set of points, e.g., a (closed) polygonal element, must take
its maximum and minimum value on the boundary.
If all its vertex values are nonnegative, so is their piecewise linear
interpolation on the elemental boundary and the function itself inside
the element.
Unfortunately, we cannot apply this theorem to the modified
(``\emph{enhanced}'') virtual element space introduced
in~\cite{Ahmad-Alsaedi-Brezzi-Marini-Russo:2013} as its functions are
no longer harmonic.
This fact poses a major issue to the design of our VEM since we need
an $\LS{2}$-like orthogonal projector for the calculation of the mass
matrix in the discretization of the time derivative term.
To address this issue, we design a different projector, which is still
computable from the degrees of freedom of the space and is orthogonal
with respect to an approximate $\LS{2}$ inner product.
We carefully characterize the approximation properties of this
operator to prove the convergence of the VEM, estimate the
approximation error and derive the convergence rate for the
\RED{refinement} in time and space.

We also prove the well-posedness of the numerical method, i.e.,
existence and uniqueness of the virtual element solution, in two
different ways.
The first proof reveals the contractive nature of the scheme, which
motivates an iterative implementation at every time step from a
practical viewpoint.
The second proof generalizes a minimization argument briefly mentioned
in~\cite{Johnson:1976} to the new virtual element framework proposed
in this work and establishes a clear connection between the VEM and
the minimization of quadratic functionals.

To carry out the theoretical analysis and prove the convergence of the
VEM, we investigate how the virtual element reformulation impacts on
the original convergence proof of Reference~\cite{Johnson:1976}.
A major ingredient of the latter is the existence of a nonnegative
quasi-interpolation operator for functions that are only
$\HS{1}$-regular as, for example, the derivative in time of the
parabolic inequality solution.
To address this point, we generalize the construction of such operator
in~\cite{Johnson:1976}, so that it can work on polygonal elements with
the desired nonnegativity property.
Finally, we identify the new terms that arise from the
``\emph{variational crime}'' introduced by the virtual element method
and provide an upper bound for all of them.

The numerical experiments confirm the validity of our approach by
solving a manufactured solution problem on very general meshes
including distorted square elements, nonconvex elements and Voronoi
tesselations.
The experimental convergence rates reflects the convergence rates
expected from the theoretical analysis.

The outline of the paper is as follows.
In the rest of this section, we introduce some background material
from functional analysis and the notation used in the paper.
In Section~\ref{sec:model:problem}, we discuss the continuous weak
formulation of the mathematical model.
In Section~\ref{sec:VEM}, we present our virtual element method for
the parabolic inequality problem.
In Section~\ref{sec:technical:lemmas}, we introduce some technical
lemmas and detail the construction of the quasi-interpolation
nonnegative operator for the convergence analysis.
In Section~\ref{sec:convergence:analysis}, we prove the convergence of
the method and derive the a priori error estimate.
In Section~\ref{sec:numerical:experiments}, we assess the performance
of the method on three different families of polygonal meshes.
In Section \ref{sec:conclusions}, we summarize our results and offer
the final remarks.

\subsection{Notation}

In the rest of this section, we introduce some background material
from functional analysis as a few basic definitions of functional
spaces, inner products, norms and seminorms.
The notation adopted in this paper is consistent with
Reference~\cite{Adams-Fournier:2003} for the Sobolev and Hilbert
spaces and Reference~\cite{Evans:1998} for the Bochner spaces.

\subsubsection{Functional spaces}
Let $\omega$ be an open, bounded, connected subset of $\REAL^{2}$.
We consider a real number $\ps$ such that $1\leq\ps<\infty$ and an
integer number $k\geq1$.
We denote the Sobolev space of the real-valued, $\ps$-integrable
functions defined on $\omega$ by $\LS{\ps}(\omega)$, and the Sobolev
space of the real-valued, essentially bounded functions defined on
$\omega$ by $\LS{\infty}(\omega)$.
We denote the subspace of functions of $\LS{\ps}(\omega)$ whose weak
derivatives of order up to $k$ are also in $\LS{\ps}(\omega)$ by
$\WS{k,p}(\omega)$.
For $\ps=2$, we prefer the notation $\HS{k}(\omega)$.
We recall that $\LS{2}(\omega)$ and $\HS{k}(\omega)$ are Hilbert
spaces when endowed with the inner products
\begin{align}
  \scal{\phi}{\psi}_{\omega}  &:=\int_{\omega}\phi(\xv)\psi(\xv)d\xv\qquad\forall\phi,\psi\in\LS{2}(\omega),\\[0.25em]
  \scal{\phi}{\psi}_{k,\omega}&:=\sum_{\ABS{\alpha}\leq\ks}\int_{\omega}\Ds^{\alpha}\phi(\xv)\Ds^{\alpha}\psi(\xv)d\xv\qquad\forall\phi,\psi\in\HS{k}(\omega),\quad\ks\geq1,
\end{align}
and the induced norms
$\NORM{\psi}{0,\omega}=\scal{\psi}{\psi}_{\omega}^{\half}$ and
$\NORM{\psi}{k,\omega}=\scal{\psi}{\psi}_{k,\omega}^{\half}$.
All integrals must be intended in the sense of the Lebesgue
integration theory and we may use the abbreviation ``\textit{a.e.}''
for ``\textit{almost everywhere}'' 
whenever a \emph{pointwise} property holds except for a subset of
points with zero Lebesgue measure.
In the formulation of the method, $\omega$ can be a mesh element (see
the next subsection) or the whole computational domain $\Omega$.
In the last case, we omit the subscript $\Omega$ and use
$\scal{\phi}{\psi}$,
$\scal{\phi}{\psi}_{k}$,
$\NORM{\psi}{k}$ and
$\SNORM{\psi}{k}$ instead of
$\scal{\phi}{\psi}_{\Omega}$,
$\scal{\phi}{\psi}_{k,\Omega}$,
$\NORM{\psi}{k,\Omega}$ and
$\SNORM{\psi}{k,\Omega}$.

Let $\Ts>0$ be a real number and $(\Xs,\norm{\cdot}{\Xs})$ a normed
space, where $\Xs$ can be $\LS{2}(\Omega)$ or $\HS{k}(\Omega)$,
$k\geq1$.
The Bochner space $\LS{p}(0,\Ts;\Xs)$ is the space of functions $\vs$
such that the sublinear functional
\begin{align*}
  \NORM{\vs}{\LS{p}(0,\Ts;\Xs)}
  = \begin{cases}
    \displaystyle\left(\int_{0}^{T}\NORM{\vs(t)}{\Xs}^{\ps}\,dt\right)^{1\slash{p}} & 1\leq\ps<\infty,\\
    \esssup_{\ts\in[0,T]}\NORM{\vs(\ts)}{\Xs},                                    & p=\infty,
  \end{cases}
\end{align*}
is a \emph{finite} norm for almost every $t\in[0,\Ts]$.
According with this notation, we also denote the space of the
continuous functions from $[0,T]$ to $\Xs$ by $\Cs{}(0,T;\Xs)$.

\medskip
Throughout the paper, we use the letter ``$\Cs$'' to denote a strictly
positive constant that can take a different value at any occurrence.
The constant $\Cs$ is independent of the mesh size parameter $\hh$ and
the time step $\Delta\ts$ that will be introduced in the next
sections.
However, $\Cs$ may depend on the other parameters of the differential
problem and its virtual element discretization such as the domain
shape, the mesh regularity constant and the coercivity and continuity
constants of the bilinear forms used in the variational formulation.

\subsubsection{Mesh notation and regularity assumptions}

For the exposition sake, we assume that the computational domain
$\Omega$ is an open, bounded, \emph{polygonal} subset of $\REAL^2$
with Lipschitz boundary $\Gamma$.
Let $\calT=\{\Th\}_{\hh}$ be a family of mesh decompositions $\Th$ of
$\Omega$ uniquely identified by the value of the mesh size parameter
$\hh\in\calH$.
Here, $\calH$ is a suitable subset of the real line $\REAL$ having
zero as its unique accumulation point.
Every mesh $\Th$ is a collection of nonoverlapping, open, polygonal
elements denoted by $\P$ and forming a finite covering of $\Omega$,
i.e., $\overline{\Omega}=\bigcup_{\P\in\Th}\overline{\P}$.
The polygonal elements are nonoverlapping in the sense that the
intersection of the closures in $\REAL^2$ of any pair of them
$\P,\P^{\prime}\in\Th$ has area equal to zero, i.e.,
$\Abs{\overline{\P}\cap\overline{\P}^{\prime}}=0$.
Accordingly, the intersection of their boundaries
$\partial\P\cap\partial\P^{\prime}$ is either the empty set, or the
subset of common vertices, or the subset of shared edges (including
the edge vertices).
Every polygon $\P$ has a nonintersecting boundary denoted by
$\partial\P$ and formed by straight edges $\E$, area $\mP$, center of
gravity $\xvP=(\xsP,\ysP)^T$ and diameter
$\hP=\max_{\xv,\yv\in\P}\abs{\xv-\yv}$.
As usual, the maximum of the diameters of the elements in a mesh $\Th$
provides the value of the mesh size $\hh$, e.g.,
$\hh=\max_{\P\in\Th}\hP$.
Consistently with this notation, $\hE$ is the length $\mE$ of edge
$\E$ and $\xvE=(\xsE,\ysE)^T$ is the position vector of the midpoint
of edge $\E$.

\medskip
In the formulation of the VEM, we require that all the meshes $\Th$
satisfy the following \emph{mesh regularity assumption}.
\begin{description}
\item[]\ASSUM{M}{} There exists a real, \RED{strictly positive constant $\rho >0$},
  which is independent of $\hh$, such that:

  \vspace{-0.5em}
  \begin{description}
  \item\ASSUM{M}{1} every element $\P\in\Th$ is star-shaped with
    respect to a ball of radius greater than $\rho\hP$;
  \item\ASSUM{M}{2} for every element $\P\in\Th$, the length $\hE$ of
    every edge $\E\subset\partial\P$ satisfies $\hE\geq\rho\hP$.
  \end{description}
\end{description}
An admissible mesh that satisfies assumptions
\ASSUM{M}{1}-\ASSUM{M}{2} may have elements with a very general
geometric shape.
However, the star-shapedness property \ASSUM{M}{1} implies that the
polygonal elements are \emph{simply connected} subsets of $\REAL^{2}$,
and the scaling assumption \ASSUM{M}{2} implies that the elements
cannot become too skewed and the number of edges in each elemental
boundary is uniformly bounded over the whole mesh family
$\{\Th\}_{\hh}$.

\subsubsection{Polynomial spaces}

We denote the linear space of polynomials of degree $\ell=0,1$ defined
on the element $\P$ or the edge $\E$ by $\PS{\ell}(\P)$ and
$\PS{\ell}(\E)$, respectively, and we conveniently set
$\PS{-1}(\P)=\{0\}$.
Space $\PS{1}(\P)$ is the span of the \emph{scaled monomials} defined
as:
\begin{align}
  \ms_1(\xs,\ys) = 1,\qquad
  \ms_2(\xs,\ys) = \frac{\xs-\xsP}{\hP},\qquad
  \ms_3(\xs,\ys) = \frac{\ys-\ysP}{\hP}\qquad
  \forall(\xs,\ys)\in\P.
  \label{eq:scaled:monomials}
\end{align}
Similarly, $\PS{1}(\E)$ is the span of the monomials
$\mu_1(s)=1,\mu_2(s)=(s-s_{\E})\slash{\hE}$, where $s\in\E$ is a local
coordinate on edge $\E$, and $s_{\E}$ is the position of the edge
midpoint $\xvE$ in such a local cordinate system.
We let $\PS{1}(\Th)$ denote the linear space of the piecewise
discontinuous polynomials that are globally defined on $\Omega$ and
such that $\restrict{\qs}{\P}\in\PS{1}(\P)$ for all elements
$\P\in\Th$.

\medskip
In the VEM formulation, we make use of the elliptic projection
operator $\PinP{}:\HS{1}(\P)\rightarrow\PS{1}(\P)$, which is defined
on every mesh element $\P$ so that, for all $\vs\in\HS{1}(\P)$, the
linear polynomial $\PinP{}\vs$ is the solution to the variational
problem
\begin{align}
  \big(\nabla(\PinP{}\vs-\vs),\nabla\qs\big)_{\P} = 0 \quad\forall\qs\in\PS{1}(\P)
  \quad\textrm{and}\quad
  \Ps^{0,\P}\big(\PinP{}\vs-\vs) = 0.
  \label{eq:elliptic:projector}
\end{align}
In~\eqref{eq:elliptic:projector}, $\Ps^{0,\P}\vs$ is the projection of
$\vs$ onto the constant polynomials given by
\begin{align}
  \Ps^{0,\P}\vs:=\frac{1}{\ABS{\partial\P}}\int_{\partial\P}\vs(\xv)d\xv.
\end{align}
Accordingly, we define the global elliptic projection operator
$\Pin{}:\HS{1}(\Omega)\to\PS{1}(\Th)$ as the operator satisfying
$\restrict{\big(\Pin{}\vs\big)}{\P}=\PinP{}\big(\restrict{\vs}{\P}\big)$ for every
mesh element $\P$.

\medskip
For the sake of reference, we also define the orthogonal projection
operator $\PizP{}:\LS{2}(\P)\rightarrow\PS{1}(\P)$ with respect to the
inner product in $\LS{2}(\P)$, although we will not use it in the
formulation of the method.
The orthogonal projection $\PizP{}\vs$ of a function
$\vs\in\LS{2}(\P)$ is the linear polynomial solving the variational
problem:
\begin{align}
  \big(\PizP{}\vs-\vs,\qs\big)_{\P} = 0 \quad\forall\qs\in\PS{1}(\P).
  \label{eq:orthogonal:projector}
\end{align}
Accordingly, we define the global orthogonal projection operator
$\Piz{}:\LS{2}(\Omega)\to\PS{1}(\Th)$ as the operator satisfying
$\restrict{\big(\Piz{}\vs\big)}{\P}=\PizP{}\big(\restrict{\vs}{\P}\big)$ for
every mesh element $\P$.


\section{Parabolic Variational Inequality}
\label{sec:model:problem}

We let
$\calK=\big\{\,\vs\in\HSzr{1}(\Omega):\,\vs\geq0~\textrm{a.e.~in~}\Omega\,\big\}$
be the subset of the nonnegative functions in $\HSzr{1}(\Omega)$.
We also consider the positive real number $\Ts$ representing the final
integration time and the time interval $\Js=[0,\Ts]$, and introduce
the bilinear form
\begin{align}
  \as(\vs,\ws)=\int_{\Omega}\nabla\vs(\xv)\cdot\nabla\ws(\xv)d\xv
  \qquad\forall\vs,\ws\in\HS{1}(\Omega).
  \label{eq:bilinear:forms}
\end{align}
This bilinear form is coercive and continuous on $\HSzr{1}(\Omega)$.
So, there exists two real, positive constants $\alpha$ and $\Ms$ such
that that $\alpha\NORM{\vs}{1}^2\leq\as(\vs,\vs)$ and
$\as(\vs,\ws)\leq\Ms\NORM{\vs}{1}\NORM{\ws}{1}$ for all $\vs$, $\ws$
in $\HSzr{1}(\Omega)$.
We search the solution $\us(\ts)$ to the parabolic variational
inequality problem for a given right-hand side source term $\fs$ and
initial state $\us_0$, which reads as

\medskip
\textit{Find $\us(\ts):\Js\to\calK$ such that, for almost every
  $\ts\in\Js$ it holds that}
\begin{subequations}
  \begin{align}
    &\Scal{\frac{\partial\us}{\partial\ts}}{\vs-\us}+\as(\us,\vs-\us)
    \geq \scal{\fs}{\vs-\us}\quad\forall\vs\in\calK,
    \label{eq:parb:ineq:A}\\[0.5em]
    &~\us(0)=\us_0.
    \label{eq:parb:ineq:B}
  \end{align}
\end{subequations}
The solution $\us$ exists and is unique~\cite{Brezis:1972} under the
assumptions
\begin{description}
\item[]\ASSUM{A}{1} $\fs\in\Cs{}\big(J;\LS{\infty}(\Omega)\big)$;
\item[]\ASSUM{A}{2} $\partial\fs\slash{\partial\ts}\in\LS{2}\big(J;\LS{\infty}(\Omega)\big)$;
\item[]\ASSUM{A}{3} $\us_0\in\WS{2,\infty}(\Omega)\cap\calK$.
\end{description}
In particular, if assumptions \ASSUM{A}{1}-\ASSUM{A}{3} are true,
solution $\us$ is such that:
\begin{subequations}
  \begin{align}
    &\us\in\LS{\infty}\big(0,T;\WS{2,p}(\Omega)\big)\qquad\textrm{for~}1\leq\ps<\infty,
    \label{eq:regularity:A}
    \\[0.5em]
    &\frac{\partial\us}{\partial\ts}
    \in\LS{\infty}\big(0,T;\HSzr{1}(\Omega)\big)\cap\LS{\infty}\big(0,T;\LS{\infty}(\Omega)\big),
    \label{eq:regularity:B}
    \\[0.5em]
    &\Scal{\frac{\partial^{+}\us}{\partial\ts}}{\vs-\us}+\as(\us,\vs-\us)\geq\scal{\fs}{\vs-\us}
    \quad\forall\vs\in\calK,\,\ts\in\Js,
    \label{eq:regularity:C}
  \end{align}
\end{subequations}
where $\partial^{+}\us\slash{\partial\ts}$ denotes the right-hand
derivative of $\us$ with respect to $\ts$.
Moreover, $\us$ satisfies the partial differential equations
\begin{subequations}
  \begin{align}
    &\frac{\partial^{+}\us}{\partial\ts}=\Delta\us+\fs    \quad\text{a.e.~on~}\Omega^{+}(\ts),
    \label{eq:aux:eqn:A}
    \\[0.5em]
    &\frac{\partial^{+}\us}{\partial\ts}=\text{max}(\fs,0)\quad\text{a.e.~on~}\Omega^{0}(\ts),
    \label{eq:add:eqn:B}
  \end{align}
\end{subequations}
where, for almost every $\ts\in\Js$,
$\Omega^{+}(\ts)=\big\{\xv\in\Omega:\us(\xv,\ts)>0\big\}$ and
$\Omega^{0}(\ts)=\big\{\xv\in\Omega:\us(\xv,\ts)=0\big\}$.

\medskip
Finally, we partition the time interval $[0,\Ts]$ into $\Ns$ equally
spaced subintervals $\Js_n=\big[\ts^{n},\ts^{n+1}\big]$ having size
$\Delta\ts=\ts^{n+1}-\ts^{n}=\Ts\slash{\Ns}$, and let $\ms(\Gamma_n)$
denote the area of the set
\begin{align}
  \Gamma_n =
  \underset{\ts\in\Js_n}{\cup}\Omega^{+}(\ts)\cup\Omega^{+}(\ts^{n+1})
  \setminus\overline{\Omega^{+}(\ts)\cap\Omega^{+}(\ts^{n+1})}.
  \label{eq:Gamma-n:def}
\end{align}
Our last assumption is that
\begin{description}
\item[]\ASSUM{A}{4} $~\sum_{n=0}^{N-1}\ms(\Gamma_n)\leq\Cs$ for some
  real, positive constant $\Cs$ independent of $\hh$ and $\Delta\ts$.
\end{description}
This assumption together with \ASSUM{A}{1}-\ASSUM{A}{3} will be used
in the convergence analysis of the method that we perform in
Section~\ref{sec:convergence:analysis}.


\section{Virtual element approximation}
\label{sec:VEM}

Let $\Vsh$ be a conforming finite dimensional subspace of
\RED{$\HSzr{1}(\Omega)$} that will be referred to as the \emph{virtual
element space}.
Let $\msh(\cdot,\cdot),\,\ash(\cdot,\cdot):\Vsh\times\Vsh\to\REAL$ be
the virtual element approximation of the $\LS{2}$ inner product
$\scal{\cdot}{\cdot}$ and the bilinear form $\as(\cdot,\cdot)$.
Let $\fsh$ be the element of $(\Vsh)^{\prime}$, the dual space of
$\Vsh$, such that $\scal{\fsh}{\cdot}:\Vsh\to\REAL$ is \RED{a} virtual
element approximation of the linear functional $\scal{\fs}{\cdot}$ (we
use the same symbol $\fsh$ to denote the Ritz representative of $\fsh$
in $\Vsh$).
Then, we introduce the finite-dimensional subset
$\calKh=\Vsh\cap\calK=\big\{\vsh\in\Vsh:\vsh\geq0\textrm{~in~}\Omega\big\}$
of the virtual element functions that are nonnegative in $\Omega$.
We denote the evaluation of a time-dependent quantity $\ws(t)$ at
$\ts^{n}$ by $\wss{n}=\ws(\ts^{n})$, and define the discrete
difference operator
$\partial\wss{n}=\big(\wss{n+1}-\wss{n}\big)\slash{\Delta\ts}$, which
provides the time variation of $\{\ws(\ts^{n})\}_{n}$ in the time
interval $\big[\ts^{n},\ts^{n+1}\big]$.

\medskip
The virtual element approximation $\Ush^{n}$ to $\us(\ts^{n})$ is the
solution of the following discrete problem:
\textit{Find $\{\Ush^{n}\}_{n=0,\ldots,\NT}$ with $\Ush^{n}\in\calKh$
  for every $n=0,1,\ldots,\NT$ such that}
\begin{align}
  \msh\big(\partial\Ush^n,\vsh-\Ush^{n+1}\big) + \ash\big(\Ush^{n+1},\vsh-\Ush^{n+1}\big)
  \geq\Scal{\fsh^{n+1}}{\vsh-\Ush^{n+1}},
  \label{eq:VEM:A}
\end{align}
\textit{for every $\vsh\in\calKh$ with the initial solution field
  $\Ush^0$ satisfying}
\begin{align}
  & \Norm{\Ush^0-\us_0}{0}\leq\Cs\hh.
  \label{eq:VEM:B}
\end{align}

This section is devoted to the definition of $\Vsh$, the construction
of the bilinear forms $\msh(\cdot,\cdot)$ and $\ash(\cdot,\cdot)$ and
the linear functional $(\fsh,\cdot)$, and the characterization of
their approximation properties.
\RED{
Furthermore, a possible choice of the initial approximation of
$\us_0$, i.e., $\Ush^0$, which satisfies~\eqref{eq:VEM:B}, is provided
by could be chosen as $\Ush^0=\Ih\us_0$, where $\Ih$ is the
interpolation operator that will be defined in
Section~\ref{subsec:nonnegative:quasi-interpolation:operator}.
}

\subsection{Virtual element spaces}
\label{subsec:VEM:spaces}
Following~\cite{BeiraodaVeiga-Brezzi-Cangiani-Manzini-Marini-Russo:2013},
we define the virtual element space $\Vsh(\P)$ on every element
$\P\in\Th$ as
\begin{align}
  \Vsh(\P)=\Big\{\,
  \vsh\in\HS{1}(\P)\cap\CS{}(\overline{\P}):
  \restrict{\vsh}{\partial\P}\in\CS{}(\partial\P),\;
  \restrict{\vsh}{\E}\in\PS{1}(\E)\;\forall\E\in\partial\P,\;
  \Delta\vsh=0\textrm{~in~}\P
  \,\Big\}.
  \label{eq:VEM:space:local}
\end{align}
The global virtual element space $\Vsh$ is given by gluing together in
a conforming way the elemental spaces $\Vsh(\P)$:
\begin{align}
  \Vsh:=\Big\{
  \vsh\in\HSzr{1}(\Omega):
  \restrict{\vsh}{\P}\in\Vsh(\P)\,\,\forall\P\in\Th
  \Big\}.
  \label{eq:VEM:space:global}
\end{align}
On every element $\P\in\Th$, we consider the subset of the nonnegative
virtual element functions:
\begin{align}
  \calKh(\P)
  =\Big\{\,\vsh\in\Vsh(\P):\vsh\geq0\,\textrm{~in~}\P\Big\}
  \subset\HSzr{1}(\P).
\end{align}
It is immediate to see that $\restrict{\vsh}{\P}\in\calKh(\P)$ for all
$\P\in\Th$ if and only if $\vsh\in\calKh$, since
$\calKh=\Vsh\cap\calK$ is the subset of the nonnegative virtual
element functions globally defined on $\Omega$.

\medskip
A virtual element function $\vsh$ is uniquely characterized in every
element $\P$ by its values at the elemental vertices, so that we can
take such values as the degrees of freedom of the method.
A proof of this unisolvence property is found
in~\cite{BeiraodaVeiga-Brezzi-Cangiani-Manzini-Marini-Russo:2013}.
The degrees of freedom of the functions in the global space $\Vsh$ and
its subset $\calKh$ are given by collecting the values at all the mesh
vertices.
Their unisolvence in $\Vsh$ follows from their unisolvence in each
elemental space.
Moreover, a function $\vsh\in\calKh(\P)$ also belongs to $\Vsh(\P)$,
so it is uniquely defined by its vertex values, but these values must
be nonnegative to reflect the property that $\vsh(\xv)\geq0$ for every
$\xv\in\overline{\P}$.
This property, which is crucial in the construction of our VEM, is
stated in the following lemma.
\begin{lemma}[Nonnegative characterization of $\calKh(\P)$]
  \label{lemma:VEM:nonnegative-subset}
  Let $\P$ denote an element of mesh $\Th$ satisfying the mesh
  assumptions \ASSUM{M}{1}-\ASSUM{M}{2}.
  Then, a virtual element function $\vsh\in\Vsh(\P)$ belongs to
  $\calKh(\P)$ if and only if its values at the vertices of $\P$ are
  nonnegative.
\end{lemma}
\begin{proof}
  The evaluation of a nonnegative function $\vsh\in\Vsh(\P)$ at the
  vertices of $\P$ is obviously nonnegative.
  In turn, the edge trace $\restrict{\vsh}{\E}$ for each edge $\E$ is
  nonnegative if the values of $\vsh$ at the vertices of
  $\E\subset\partial\P$ are nonnegative since the trace is given by the
  linear interpolation of such vertex values.
  Then, the lemma is a consequence of the Maximum and Minimum
  Principle Theorem, see~\cite{Gilbarg-Trudinger:2001}, which implies
  that all nonconstant harmonic functions defined on the nonempty
  compact subset $\overline{\P}$ of $\REAL^2$ attains their maximum
  and minimum values on the boundary of $\P$.
  \ENDPROOF
\end{proof}
This result is readily extended to the whole set $\calKh$ in the next
corollary.
\begin{corollary}
  \label{lemma:VEM:nonnegative-subset:global}
  Under mesh assumptions \ASSUM{M}{1}-\ASSUM{M}{2}, a virtual element
  function $\vsh\in\Vsh$ belongs to $\calKh$ if and only if its values
  at the mesh vertices are nonnegative.
\end{corollary}
\begin{proof}
  The assertion of the lemma trivially follows from the previous lemma
  and the definition of the degrees of freedom of a virtual element
  function in the subset $\calKh$.
  \ENDPROOF
\end{proof}
  
\medskip
The polynomial space $\PS{1}(\P)$ is a linear subspace of $\Vsh(\P)$
and the subset of the nonnegative linear polynomials must belong to
$\calKh(\P)$.
Moreover, Lemma~\ref{lemma:VEM:nonnegative-subset} implies that a
linear polynomial whose vertex values are nonnegative must be
nonnegative.

A major property of the elemental space $\Vsh(\P)$ is that the
elliptic projection $\PinP{}\vsh$ of the virtual element function
$\vsh$ defined in~\eqref{eq:elliptic:projector} is computable from the
degrees of freedom of $\vsh$.
In the spirit of the VEM, we will use this projection operator to
define the discrete bilinear form $\ash(\cdot,\cdot)$, see the next
subsection.
Instead, the orthogonal projection $\PizP{}\vsh$ is noncomputable from
the degrees of freedom of the virtual element function $\vsh$.
Following~\cite{Ahmad-Alsaedi-Brezzi-Marini-Russo:2013}, we could
consider the ``enhanced'' virtual element space:
\begin{multline}
  \Vsht(\P)=\Big\{\,
  \vsh\in\HS{1}(\P):
  \restrict{\vsh}{\partial\P}\in\CS{0}(\partial\P),\;
    \restrict{\vsh}{\E}\in\PS{1}(\E)\;\forall\E\in\partial\P,\;\\
    \Delta\vsh\in\PS{1}(\P),\;
    \scalP{\vsh-\PinP{}\vsh}{\qs}=0\;\forall\qs\in\PS{1}(\P)
    \,\Big\}.
    \label{eq:VEM:space:local:enhancement}
\end{multline}
In such a space, the orthogonal projection $\PizP{}\vsh$ coincides
with the elliptic projection $\PinP{}\vsh$.
However, a fundamental property of our construction is that a virtual
element function with all positive (nonnegative) values at the
vertices of $\P$ must be positive (nonnegative) in $\P$.
We can readily prove this property for the harmonic functions of space
\eqref{eq:VEM:space:local} by resorting to the Maximum and Minimum
Principle Theorem~\cite{Gilbarg-Trudinger:2001} as in the proof of
Lemma~\ref{lemma:VEM:nonnegative-subset}, but not for the nonharmonic
functions of space $\Vsht(\P)$
in~\eqref{eq:VEM:space:local:enhancement}.
So, to define the bilinear form $\msh(\cdot,\cdot)$ we need to use a
different polynomial reconstruction, which is based on the alternative
projection operator of the next subsection.

\medskip
The next two lemmas establish the local approximation properties of
the virtual element interpolation operator and a polynomial
approximation operator.
These approximation properties hold under the mesh regularity
assumptions~\ASSUM{M}{1}-\ASSUM{M}{2},
cf.~\cite{BeiraodaVeiga-Brezzi-Cangiani-Manzini-Marini-Russo:2013}.
We omit their proof as they are standard results from the literature.
\begin{lemma}
  \label{lemma:interpolation}
  Let $\P$ be a polygonal element of a mesh $\Th$ satisfying
  assumptions~\ASSUM{M}{1}-\ASSUM{M}{2}.
  Then, there exists a real, positive constant $\Cs$ such that for all
  $\vs\in\HS{2}(\P)$ the virtual element interpolant
  $\vsI\in\Vsh(\P)$, which is the function in $\Vsh(\P)$ with the same
  vertex values of $\vs$, is such that
  \begin{align}
    \norm{\vs-\vsI}{0,\P}+\hP\snorm{\vs-\vsI}{1,\P}
    \leq~\Cs\hP^2\snorm{\vs}{2,\P}.
    \label{eq:interpolation}
  \end{align}
  The constant $\Cs$ is independent of the local mesh size $\hP$ but
  may depend on the mesh regularity constant $\rho$.
\end{lemma}
We outline that if a function
$\vs\in\HS{2}(\P)\cap\CS{}(\overline{\P})$ is nonnegative in
$\overline{\P}$, than its interpolant $\vsI\in\Vsh(\P)$ must also be
nonnegative as a consequence of
Lemma~\ref{lemma:VEM:nonnegative-subset}, and it belongs to
$\calKh(\P)$.
In Section~\ref{sec:technical:lemmas}, we discuss the construction of
a nonnegative quasi-interpolation operator since in the convergence
analysis of Section~\ref{sec:convergence:analysis} we must cope with
functions that are only $\HS{1}$-regular.

\begin{lemma}
  \label{lemma:projection}
  Let $\P$ be a polygonal element of a mesh $\Th$ satisfying
  assumptions~\ASSUM{M}{1}-\ASSUM{M}{2}.
  Then, there exists a real, positive constant $\Cs$ such that for all
  $\vs\in\HS{m}(\P)$, $m=1,2$, there exists a polynomial functions
  $\vs_{\pi}\in\PS{1}(\P)$ such that
  \begin{align}
    \norm{\vs-\vs_{\pi}}{0,\P} + \hP\snorm{\vs-\vs_{\pi}}{1,\P}\leq\Cs\hP^{m}\snorm{\vs}{m,\P}.
    \label{eq:projection}
  \end{align}
  The constant $\Cs$ is independent of the local mesh size $\hP$ but
  may depend on the mesh regularity constant $\rho$.
\end{lemma}

\subsection{The projection operator $\PiLSP{}$}
Consider the discrete inner product in $\Vsh(\P)$:
\begin{align}
  \qScalP{\vsh}{\wsh} = \mP\sum_{i=1}^{\NVP}\vsh(\xv_i)\wsh(\xv_i)
  \quad\forall\vsh,\wsh\in\Vsh(\P),
  \label{eq:discrete:inner:product}
\end{align}
where $\NVP$ is the number of vertices of $\P$, and
$\xv_i=(\xs_i,\ys_i)^T$, $i=1,\ldots,\NVP$, is the coordinate vector of
the $i$-th vertex of element $\P$.
Then, for every $\vsh\in\Vsh(\P)$, we define $\PiLSP{}\vsh$ as the
linear polynomial that solves the projection problem:
\begin{align}
  \qSCALP{\vsh-\PiLSP{}\vsh}{\qs} = 0 \quad\forall\qs\in\PS{1}(\P).
  \label{eq:PiLSP:def}
\end{align}
This projection operator is computable from the degrees of freedom of
$\vsh$.
Indeed, we consider the expansion of $\PiLSP{}\vsh$ on the scaled
monomial basis of $\PS{1}(\P)$:
\begin{align}
  \PiLSP{}\vsh = \zeta_1\ms_1 + \zeta_2\ms_2 + \zeta_3\ms_3,
  \label{eq:PiLSP:expansion}
\end{align}
with $\zeta_i\in\REAL$, $i=1,2,3$.
Then, we introduce matrix $\matD$, which collects the degrees of
freedom of $\ms_i$ on its $i$-th column, so that
\begin{align}
  \matD =
  \left[
    \begin{array}{ccc}
      \ms_1(\xv_1)     & \ms_2(\xv_1)     & \ms_3(\xv_1)    \\[0.5em]
      \ms_1(\xv_2)     & \ms_2(\xv_2)     & \ms_3(\xv_2)    \\[0.5em]
      \vdots & \vdots                                    \\[0.5em]
      \ms_1(\xv_{\NVP}) & \ms_2(\xv_{\NVP}) & \ms_3(\xv_{\NVP}) \\[0.5em]
    \end{array}
  \right]
  =
  \left[
    \begin{array}{ccc}
      1      & \quad\dfrac{\xs_{1}   -\xsP}{\hP}   & \quad\dfrac{\ys_{1}    -\ysP}{\hP}\\[0.75em]
      1      & \quad\dfrac{\xs_{2}   -\xsP}{\hP}   & \quad\dfrac{\ys_{2}    -\ysP}{\hP}\\[0.75em]
      \vdots & \vdots                             & \vdots                            \\[0.75em]
      1      & \quad\dfrac{\xs_{\NVP} -\xsP}{\hP}   & \quad\dfrac{\ys_{\NVP} -\ysP}{\hP}
    \end{array}
    \right].
  \label{eq:matD:def}
\end{align}
A straightforward calculation allows us to
reformulate~\eqref{eq:PiLSP:def} in the vector form:
\begin{align}
  \big(\matD^T\matD\big)\zetav = \matD\vvh,
  \label{eq:PiLSP:vector:def}
\end{align}
where $\zetav=(\zeta_1,\zeta_2,\zeta_3)^T$ and
$\vvh=\big(\vsh(\xv_1),\vsh(\xv_2),\ldots,\vsh(\xv_{\NVP})\big)^T$.
We note that the $3\times3$-sized matrix $\matD^T\matD$ is such that
$\RANK(\matD^T\matD)=\RANK(\matD)=3$, so it is nonsingular.
Therefore, the solution $\zetav$ of~\eqref{eq:PiLSP:vector:def} is
given by $\zetav=\big(\matD^T\matD\big)^{-1}\matD^T\vvh$.

\medskip
The next lemma characterizes the properties of the projection operator
$\PiLSP{}$.
\begin{lemma}[Properties of $\PiLSP{}$]
  \label{lemma:PiLS:properties}
  Let $\P$ be an element of mesh $\Th$ satisfying mesh
  assumptions~\ASSUM{M}{1}-\ASSUM{M}{2} and $\PiLSP{}$ the projection
  operator defined in~\eqref{eq:PiLSP:def}.
  Then,
  \begin{itemize}
  \item[$(i)$] $\PiLSP{}$ is invariant on the linear polynomials,
    i.e., $\PiLSP{}\qs=\qs$ for every $\qs\in\PS{1}(\P)$, and, thus,
    idempotent, i.e., $\big(\PiLSP{}\big)^2=\PiLSP{}$;
  \item[$(ii)$] $\PiLSP{}$ is bounded in $\LS{2}(\P)$, i.e.,
    $\Norm{\PiLSP{}\vsh}{0,\P}\leq\Cs\NORM{\vsh}{0,\P}$ for every
    $\vsh\in\Vsh(\P)$ and some real, positive constant $\Cs$
    independent of \RED{$\hP$}.
  \end{itemize}
\end{lemma}
\begin{proof}
  $(i)$. Since $\PiLSP{}$ is a linear operator, to prove that it
 is invariant on the linear polynomials, we only need to prove that it
  is invariant on the scaled monomials~\eqref{eq:scaled:monomials},
  i.e., $\PiLSP{}\ms_i=\ms_i$, $i=1,2,3$.
  We note that the vector collecting the degrees of freedom of $\ms_i$
  coincides with the $i$-th column of matrix $\matD$, which we
  indicate by $\textbf{col}(\ms_i)$.
  Let $\ev_i$ be the vector of the canonical basis of $\REAL^{\NVP}$
  having the $i$-th entry equal to $1$ and all other entries equal to
  $0$, so that $\textbf{col}(\ms_i)=\matD\ev_i$.
  The coefficient vector $\zetav_i$ of $\PiLSP{}\ms_i$ in
  expansion~\eqref{eq:PiLSP:expansion} is given by a straightforward
  application of the projection matrix $(\matD^T\matD)^{-1}\matD^T$
  to $\textbf{col}(\ms_i)$:
  \begin{align*}
    \zetav_i
    = (\matD^T\matD)^{-1}\matD^T\textbf{col}(\ms_i)
    = (\matD^T\matD)^{-1}\matD^T\matD\ev_i
    = \ev_i.
  \end{align*}
  Substituting $\zetav_i=\ev_i$ in~\eqref{eq:PiLSP:expansion} yields
  $\PiLSP{}\ms_i=\ms_i$.
  Then, the invariance of $\PiLSP{}$ on the linear polynomials implies
  that
  $\big(\PiLSP{}\big)^2\vsh=\PiLSP{}\big(\PiLSP{}\vsh\big)=\PiLSP{}\vsh$
  for all $\vsh\in\Vsh(\P)$ since $\PiLSP{}\vsh\in\PS{1}(\P)$.

  \medskip
  \noindent
  $(ii)$. Finally, we are left to prove that $\PiLSP{}$ is a bounded
  operator with an inequality constant that is independent of $\hP$.
  Consider the discrete norm
  \begin{align}
    \TNORM{\vsh}{\P}^2
    = \qScalP{\vsh}{\vsh}
    = \mP\sum_{i=1}^{\NVP}\ABS{\vsh(\xv_{i})}^2
    = \mP\abs{\vvh}^2,
    \label{app:eq:discrete:norm}
  \end{align}
  which is induced by the discrete inner
  product~\eqref{eq:discrete:inner:product}.
  We observe that $\PiLSP{}$ is a continuous operator, i.e.,
  $\Tnorm{\PiLSP{}\vsh}{\P}\leq\TNORM{\vsh}{\P}$ for every
  $\vsh\in\Vsh(\P)$.
  Indeed, $\PiLSP{}$ is the orthogonal projection operator with
  respect to the inner product~\eqref{eq:discrete:inner:product} and
  its operator norm is
  $\sup_{\vsh\in\Vsh(\P)\setminus\{0\}}\Tnorm{\PiLSP{}\vsh}{\P}\slash{\TNORM{\vsh}{\P}}=1$.
  The norm defined in~\eqref{app:eq:discrete:norm} is spectrally
  equivalent to the $\LS{2}$ norm, so that there exist two strictly
  positive constant $\xi_*$ and $\xi^*$ such that
  \begin{align}
    \xi_*\NORM{\vsh}{0,\P}\leq\TNORM{\vsh}{\P}\leq\xi^*\NORM{\vsh}{0,\P}
    \quad\forall\vsh\in\Vsh(\P).
    \label{app:eq:norm:equivalence}  
  \end{align}
  The two norms $\NORM{\vsh}{0,\P}$ and $\TNORM{\vsh}{\P}$ have the
  same scaling with respect to $\hP$ because of the explicit
  dependence of norm $\TNORM{\,\cdot\,}{\P}$ on $\mP$.
  Therefore, the two constants $\xi_*$ and $\xi^*$ may depend on the
  geometric shape of $\P$ but must be independent of $\hP$.
  Then, we use the left inequality of~\eqref{app:eq:norm:equivalence},
  the continuity of $\PiLSP{}$, 
  and the right inequality of~\eqref{app:eq:norm:equivalence}, and we find
  that
  \begin{align*}
    \Norm{\PiLSP{}\vsh}{0,\P}
    \leq (\xi_*)^{-1}\Tnorm{\PiLSP{}\vsh}{\P} 
    \leq (\xi_*)^{-1}\Tnorm{\vsh}{\P} 
    \leq \frac{\xi^*}{\xi_*} \NORM{\vsh}{0,\P}.
  \end{align*}
  We complete the proof by setting $\Cs=(\xi^*\slash{\xi_*})$ and
  noting that this constant is independent of $\hP$.
  \ENDPROOF
\end{proof}

To characterize the approximation properties of the projection
operator $\PiLSP{}$, we apply sistematically the result
in~\cite[Theorem~2]{Bramble-Hilbert:1970}, which will be referred
hereafter as the \emph{Bramble-Hilbert lemma}.
For future reference in our paper, we report the statement of this
result below, with a few, very minor changes to adapt it to our
notation and setting.
In the next subsection, we will also use the Bramble-Hilbert lemma to
characterize the approximation of the right-hand side functional
$(\fs,\cdot)$ by $(\fsh,\cdot)$,
cf. Lemma~\ref{lemma:fs:approximation}.

\begin{lemma}[Bramble-Hilbert lemma]
  \label{lemma:Bramble-Hilbert}
  Let $\P$ be a polygonal element with diameter $\hP$ satisfying mesh
  assumptions \ASSUM{M}{1}-\ASSUM{M}{2}.
  Let $\Fs$ be a linear functional on $\WS{k,p}(\P)$ which satisfies
  \begin{itemize}
  \item[$(i)$] $\ABS{\Fs(\us)}\leq\Cs\NORM{\us}{k,p,\P}$ for all
    $\us\in\WS{k,p}(\P)$ with $\Cs$ independent of $\hP$ and $\us$ and
  \item[$(ii)$] $\Fs(\ps)=0$ for all $\ps\in\PS{k-1}(\P)$.
  \end{itemize}
  Then, $\ABS{\Fs(\us)}\leq\Cs_1\hP^{k}\SNORM{\us}{k,p,\P}$ for all
  $\us\in\WS{k,p}(\P)$ with $\Cs_1$ independent of $\hP$ and $\us$.
\end{lemma}
\begin{proof}
  This lemma is an immediate consequence
  of~\cite[Theorem~2]{Bramble-Hilbert:1970}, which is set for a domain
  $R$ (with diameter $\rho$) that satisfies the \emph{strong cone
  property}, see~\cite[Section~4.6 (The cone condition)]{Adams-Fournier:2003}.
  A polygonal element $\P$ satisfying mesh assumptions
  \ASSUM{M}{1}-\ASSUM{M}{2} also satisfies such a geometric condition
  on the boundary $\partial\P$, so that we can identify $R$ with $\P$
  and $\rho$ with $\hP$.
  \ENDPROOF
\end{proof}

\begin{lemma}[1 - Approximation property of $\PiLSP{}$]
  \label{lemma:approx:property:1}
  Let $\P$ be a polygonal element satisfying mesh assumptions
  \ASSUM{M}{1}-\ASSUM{M}{2}.
  There exists a real, positive constant $\Cs$ independent of $\hP$
  such that for all virtual element functions $\vsh\in\Vsh(\P)$ and
  polynomials $\qs\in\PS{1}(\P)$ it holds that
  \begin{align}
    \Abs{ \scalP{\vsh-\PiLSP{}\vsh}{\qs} }
    \leq \Cs\,\hP^2\SNORM{\qs}{1,\P}\SNORM{\vsh}{1,\P}.
    \label{eq:approx:property:1}
  \end{align}
\end{lemma}
\begin{proof}
  Let $\PiLSsP{}:\Vsh(\P)\to\Vsh(\P)$ denote the adjoint operator of
  $\PiLSP{}$ with respect to the inner product in $\LS{2}(\P)$, which
  is formally defined as
  \begin{align}
    \ScalP{\PiLSsP{}\vsh}{\wsh}
    =
    \ScalP{\vsh}{\PiLSP{}\wsh}
    \quad\forall\vsh,\wsh\in\Vsh(\P).
  \end{align}
  This operator projects onto the orthogonal complement of
  $\KER(\PiLSP{})=\big\{\vsh\in\Vsh(\P)\mid\PiLSP{}\vsh=0~\textrm{in}~\P\big\}$.
  In fact, from its definition and the second property in $(i)$ of
  Lemma~\ref{lemma:PiLS:properties}, i.e., $(\PiLSP{})^2=\PiLSP{}$, we
  immediately see that
  \begin{align*}
    \ScalP{\PiLSsP{}\wsh}{\vsh-\PiLSP{}\vsh}
    = \ScalP{\wsh}{\PiLSP{}\big(\vsh-\PiLSP{}\vsh\big)}
    = \ScalP{\wsh}{\PiLSP{}\vsh-\big(\PiLSP{}\big)^2\vsh}
    = 0,
  \end{align*}
  which holds for all $\vsh,\wsh\in\Vsh(\P)$.
  Then, we note that
  $\scalP{\PiLSsP{}\qs}{1}=\scalP{\qs}{\PiLSP{}(1)}=\scalP{\qs}{1}$
  for all $\qs\in\PS{1}(\P)$ since $\PiLSP{}(1)=1$ from property $(i)$
  of Lemma~\ref{lemma:PiLS:properties}.
  Therefore, for any linear polynomial $\qs$, the cell average of
  $\qs$ and $\PiLSsP{}\qs$, respectively denoted by $\overline{\qs}$
  and $\overline{\PiLSsP{}\qs}$ are equal.

  \medskip
  For any linear polynomial $\qs$ defined on $\P$, we now consider the
  linear functional $\FsP{\qs}(\cdot):\Vsh(\P)\to\REAL$ given by
  $\FsP{\qs}(\vsh)=\scalP{\vsh-\PiLSP{}\vsh}{\qs}\slash{\Norm{\qs-\PiLSsP{}\qs}{0,\P}}$.
  The continuity of $\scalP{\cdot}{\cdot}$ implies the boundedness of
  $\FsP{\qs}(\cdot)$, which is condition $(i)$ in
  Lemma~\ref{lemma:Bramble-Hilbert},.
  In fact, it holds that 
  \begin{align}
    \ABS{ \scalP{\vsh-\PiLSP{}\vsh}{\qs} }
    &= \ABS{ \scalP{\vsh-\PiLSP{}\vsh}{\qs-\PiLSsP{}\qs}}
    \leq \Norm{\vsh-\PiLSP{}\vsh}{0,\P}\,\Norm{\qs-\PiLSsP{}\qs}{0,\P}
    \nonumber\\[0.5em]
    &
    \leq \big( \norm{\vsh}{0,\P} + \Norm{\PiLSP{}\vsh}{0,\P} \big)\,\Norm{\qs-\PiLSsP{}\qs}{0,\P}
    \leq (1+\xi^*/\xi_*)\norm{\vsh}{0,\P}\,\Norm{\qs-\PiLSsP{}\qs}{0,\P},
    \label{eq:PiLSP:approx:property:proof:10}
  \end{align}
  where $\xi_*$ and $\xi^*$ are the constants of the equivalence
  relation~\eqref{app:eq:norm:equivalence}.
  Inequality~\eqref{eq:PiLSP:approx:property:proof:10} immediately
  implies that $\FsP{\qs}(\vsh)\leq(1+\xi^*/\xi_*)\NORM{\vsh}{0,\P}$.
  Property $(i)$ of Lemma~\ref{lemma:PiLS:properties} implies
  that $\FsP{\qs}(\ps)=0$ for all $\ps\in\PS{1}(\P)$, which is
  condition $(ii)$ of Lemma~\ref{lemma:Bramble-Hilbert}.
  The Bramble-Hilbert lemma with $k=1$ and $p=2$ implies that
  \begin{align*}
    \ABS{ \FsP{\qs}(\vsh) } \leq \Cs_1\hP\SNORM{\vsh}{1,\P}
    \quad\forall\vsh\in\Vsh(\P),
  \end{align*}
  and, consequently,
  \begin{align}
    \ABS{ \scalP{\vsh-\PiLSP{}\vsh}{\qs} }
    \leq \Cs_1\hP\SNORM{\vsh}{1,\P}\Norm{\qs-\PiLSsP{}\qs}{0,\P}
    \quad\forall\vsh\in\Vsh(\P).
    \label{eq:proof:discrepancy:inner:product:00}
  \end{align}

  \medskip
  To complete the proof of the lemma, we are left to estimate
  $\Norm{\qs-\PiLSsP{}\qs}{0,\P}$.
  To this end, we add and subtract
  $\overline{\qs}=\overline{\PiLSsP{}\qs}$ and use the triangular
  inequality to find that
  \begin{align}
    \Norm{\qs-\PiLSsP{}\qs}{0,\P}
    &\leq
    \NORM{\qs-\overline{\qs}}{0,\P} +
    \Norm{\overline{\PiLSsP{}\qs}-\PiLSsP{}\qs}{0,\P}
    \leq \Cs\hP\SNORM{\qs}{1,\P}.
    \label{eq:proof:discrepancy:inner:product:05}
  \end{align}
  \RED{
  In~\eqref{eq:proof:discrepancy:inner:product:05} we used the
  inequality $\snorm{\PiLSsP{}\qs}{1,\P}\leq\Cs\SNORM{\qs}{1,\P}$,
  which is still a consequence of the fact that $\PiLSsP{}$ is a
  bounded operator and the equivalence of (semi)norms with the same
  kernel in finite dimensional spaces.
  }
  The assertion of the lemma follows by
  applying~\eqref{eq:proof:discrepancy:inner:product:05}
  to~\eqref{eq:proof:discrepancy:inner:product:00}.
  \ENDPROOF
\end{proof}

\begin{lemma}[2 - Approximation property of $\PiLSP{}$]
  \label{lemma:approx:property:2}
  Let $\P$ be a polygonal element satisfying mesh assumptions
  \ASSUM{M}{1}-\ASSUM{M}{2}.
  Then, there exists a real, positive constant $\Cs$ independent of
  $\hh$ such that for all virtual element functions $\vsh\in\Vsh$ it
  holds that
  \begin{align}
    \Norm{\vsh-\PiLSP{}\vsh}{0,\P}\leq \Cs\hP\snorm{\vsh}{1,\P}.
    \label{eq:approx:property:2}
  \end{align}
\end{lemma}
\begin{proof}
  Let $\vsh\in\Vsh$ and consider its restriction to the element
  $\P\in\Th$.
  Consider the linear functional
  $\FsP{\ws}(\vsh)=\ScalP{\vsh-\PiLSP{}\vsh}{\ws}\slash{\NORM{\ws}{0,\P}}$
  for some given function $\ws\in\LS{2}(\P)\setminus\{0\}$.
  Then, condition $(i)$ of Lemma~\ref{lemma:Bramble-Hilbert} is
  satisfied since the application of the Cauchy-Schwarz inequality and
  the boundedness of $\PiLSP{}$ yield
  \begin{align*}
    \ABS{\FsP{\ws}(\vsh)}
    \leq \frac{\Norm{\vsh-\PiLSP{}\vsh}{0,\P}\,\NORM{\ws}{0,\P}}{\NORM{\ws}{0,\P}}
    \leq \Norm{\vsh}{0,\P}+ \Norm{\PiLSP{}\vsh}{0,\P}
    \leq (1+\xi^*/\xi_*)\NORM{\vsh}{0,\P}.
  \end{align*}
  Moreover, condition $(ii)$ of Lemma~\ref{lemma:Bramble-Hilbert} is
  satisfied since $\PiLSP{}$ is invariant on all the linear
  polynomials and, so, $\FsP{\ws}(\qs)=0$ for all $\qs\in\PS{1}(\P)$.
  Since $\vsh\in\Vsh(\P)\subset\HS{1}(\P)$, the Bramble-Hilbert lemma
  (with $p=2$ and $k=1$) yields
  \begin{align*}
    \ABS{\FsP{\ws}(\vsh)}\leq \Cs_1\hP\SNORM{\vsh}{1,\P}
    \quad\forall\vsh\in\Vsh(\P).
  \end{align*}
  Recalling the definition of the $\LS{2}(\P)$ norm and using this
  inequality we obtain the upper bound
  \begin{align*}
    \Norm{\vsh-\PiLSP{}\vsh}{0,\P}
    = \sup_{\ws\in\LS{2}(\P)\setminus\{0\}}\frac{ \ScalP{\vsh-\PiLSP{}\vsh }{\ws} }{ \NORM{\ws}{0,\P} }
    = \sup_{\ws\in\LS{2}(\P)\setminus\{0\}} \ABS{\FsP{\ws}(\vsh)}
    \leq \Cs_1\hP\SNORM{\vsh}{1,\P},
  \end{align*}
  which is the assertion of the lemma.
  \ENDPROOF
\end{proof}

\begin{remark}
  \RED{Estimate~\eqref{eq:approx:property:2} is optimal for the
    virtual element functions having a local $\HS{1}(\P)$-regularity
    and a global $\HS{1}(\Omega)$-regularity.}
  Clearly, for all functions $\vsh\in\Vsh(\P)\cap\HS{2}(\P)$, the
  Bramble-Hilbert lemma would provide an error estimates proportional
  to $\hP^2\SNORM{\vsh}{2,\P}$.
\end{remark}

\subsection{The virtual element bilinear form $\msh(\cdot,\cdot)$}
%
Now, we have all the ingredients for the construction of the discrete
bilinear form $\msh(\cdot,\cdot)$.
We assume that this bilinear form is the sum of elemental
contributions
\begin{align}
  \msh(\ush,\vsh)
  =\sum_{\P\in\Th}\mshP(\ush,\vsh),
  \label{eq:msh:global:def}
\end{align}
where we define each local term as
\begin{align}
  \mshP(\ush,\vsh)
  = \ScalP{\PiLSP{}\ush}{\PiLSP{}\vsh}
  + \sPm\Big((I-\PiLSP{})\ush,(I-\PiLSP{})\vsh\Big).
  \label{eq:msh:local:def}
\end{align}
In~\eqref{eq:msh:local:def}, the bilinear form
$\sPm(\cdot,\cdot):\Vsh\times\Vsh\to\REAL$ can be \emph{any}
computable, symmetric and positive definite bilinear form such that
\begin{align}
  \sigma_*\scalP{\vsh}{\vsh}
  \leq\sPm(\vsh,\vsh)\leq
  \sigma^*\scalP{\vsh}{\vsh}
  \quad\forall\vsh\in\Vsh(\P)\cap\KER\big(\PiLSP{}\big),
  \label{eq:requirement:Sm}
\end{align}
where
$\KER\big(\PiLSP{}\big)=\big\{\vs\in\HS{1}(\P)\cap\CS{}(\overline{\P}):\PiLSP{}\vs=0\big\}$
is the kernel of the projection operator $\PiLSP{}$, and $\sigma_*$
and $\sigma^*$ are two real, positive constants independent of $\hh$
(and $\P$).

\medskip
The discrete bilinear form $\mshP(\cdot,\cdot)$ has the two crucial
properties:

\begin{itemize}
\item[-] {\emph{stability}}: for every virtual element function
  $\vsh\in\Vsh(\P)$, the following stability inequality holds
  \begin{align}
    \mu_*\scalP{\vsh}{\vsh}
    \leq\mshP(\vsh,\vsh)\leq
    \mu^*\scalP{\vsh}{\vsh},
    \label{eq:msh:stability}
  \end{align}
  where we can set $\mu^*=\max(1,\sigma^*)$ and
  $\mu_*=\min(1,\sigma_*)$,
  cf.~\cite{BeiraodaVeiga-Brezzi-Cangiani-Manzini-Marini-Russo:2013};
  
  \medskip
\item[-] {\emph{(weak) linear consistency}}: for \emph{every pair of}
  linear polynomials $\ps,\qs\in\PS{1}(\P)$ it holds that
  \begin{align}
    \mshP(\ps,\qs) = \scalP{\ps}{\qs}.
    \label{eq:msh:weak:consistency}
  \end{align}
\end{itemize}
Property~\eqref{eq:msh:stability} is a consequence of the definition
of the bilinear form $\mshP(\cdot,\cdot)$, the stability
property~\eqref{eq:requirement:Sm} of $\sPm(\cdot,\cdot)$, and the
fact that the norm $\TNORM{\,\cdot\,}{\P}$ induced by
$\qscalP{\cdot}{\cdot}$ is spectrally equivalent to the $\LS{2}$ norm
$\NORM{\,\cdot\,}{0,\P}$ induced by $\scalP{\cdot}{\cdot}$,
cf.~\eqref{app:eq:norm:equivalence}.
Property~\eqref{eq:msh:weak:consistency} is more restrictive than the
usual consistency property of the VEM as it states the exactness of
the bilinear form $\mshP(\cdot,\cdot)$ when \emph{both} its entries
are linear polynomials.
Therefore, condition~\eqref{eq:msh:weak:consistency} is weaker than
the usual consistency property of the VEM, which states that a
discrete bilinear form must be exact if at least one of the entries
(but not necessarily both) is a linear polynomial.
For this reason, we refer to ~\eqref{eq:msh:weak:consistency} as the
\emph{weak consistency} of the method.
It is worth noting that the stronger exactness property of the VEM
holds for the discrete inner
product~\eqref{eq:discrete:inner:product}:
\begin{align*}
  \qSCALP{\PiLSP{}\vsh}{\PiLSP{}\qs}
  = \qSCALP{\PiLSP{}\vsh}{\qs}
  = \qScalP{\vsh}{\qs}
  \quad
  \forall\vsh\in\Vsh(\P),
  \,  
  \forall\qs\in\PS{1}(\P),
\end{align*}
as $\PiLSP{}$ is an orthogonal projector for $\qscalP{\cdot}{\cdot}$
but an oblique one for the regular $\LS{2}$ inner product.
Lemma~\ref{lemma:msh:local:consistency:discrepancy} at the end of this
section characterizes the "obliqueness" of $\PiLSP{}$ by proving that
the discrepancy of the consistency property scales as
$\ABS{\mshP(\vsh,\qs)-\scalP{\vsh}{\qs}}=\mP\calO(\hP^2)$ for any
sufficiently regular $\vsh$ and linear polynomial $\qs$.

\medskip
As $\mshP(\cdot,\cdot)$ is a symmetric and positive definite bilinear
form, it is an inner product and the following lemma stating its
continuity stems out of an application of the Cauchy-Schwarz
inequality.
\begin{lemma}[Continuity]
  \label{lemma:msh:continuity}
  Let $\msh(\cdot,\cdot)$ be the bilinear form defined
  by~\eqref{eq:msh:global:def}-\eqref{eq:msh:local:def}.
  Then, there exists a positive constant $\Cs$ independent of $\hh$
  such that
  \begin{align*}
    \msh(\vsh,\wsh) \leq \Cs\NORM{\vsh}{0},\NORM{\wsh}{0}
  \end{align*}
  for every $\vsh,\wsh\in\Vsh$.
\end{lemma}
\begin{proof}
  First, the left inequality of the stability condition
  \eqref{eq:msh:stability} implies that for every element $\P$, the
  symmetric bilinear form $\mshP(\cdot,\cdot)$ is coercive, and, thus,
  an inner product on $\Vsh(\P)$.
  We apply the Cauchy-Schwarz inequality and the right inequality of
  the stability condition \eqref{eq:msh:stability} to obtain
  \begin{align*}
    \mshP(\vsh,\wsh)
    \leq \big(\mshP(\vsh,\vsh)\big)^{\frac12} \big(\mshP(\wsh,\wsh)\big)^{\frac12}
    \leq \mu^* \NORM{\vsh}{0,\P}\,\NORM{\wsh}{0,\P}.
  \end{align*}
  The assertion of the lemma follows by adding all the elemental
  inequalities and using again the Cauchy-Schwarz inequality to obtain
  \begin{align*}
    \msh(\vsh,\wsh)
    &= \sum_{\P\in\Th}\mshP(\vsh,\wsh)
    \leq \mu^*\sum_{\P\in\Th}\NORM{\vsh}{0,\P}\,\NORM{\wsh}{0,\P}
    \\[0.5em]
    &\leq \mu^*
    \left(\sum_{\P\in\Th}\NORM{\vsh}{0,\P}^2\right)^{\frac12}
    \left(\sum_{\P\in\Th}\NORM{\wsh}{0,\P}^2\right)^{\frac12}
    = \mu^*\NORM{\vsh}{0}\,\NORM{\wsh}{0},
  \end{align*}
  and, finally, setting $\Cs=\mu^*$, which is independent of $\hP$.
  \ENDPROOF
\end{proof}

As previously noted, it generally holds that
$\mshP(\vsh,\qs)\neq\scalP{\vsh}{\qs}$ for a nonpolynomial function
$\vsh\in\Vsh(\P)$ and a polynomial $\qs\in\PS{1}(\P)$.
We characterize the \emph{consistency discrepancy} in the final
Lemmas~\ref{lemma:msh:local:consistency:discrepancy}
and~\ref{lemma:msh:approx:consistency:global} and prove that it
locally scales as $\mP\calO(\hP^2)$ and globally scales as
$\calO(\hh^2)$.
As we will prove in the analysis of
Section~\ref{sec:convergence:analysis}, this behavior is optimal with
respect to $\hh$.
Finally, we characterize the discrepancy in the consistency of the
virtual element bilinear functional $\mshP(\cdot,\qs)$ with respect to
the inner product $\scalP{\cdot}{\qs}$ for any linear polynomial $\qs$
that is due to the use of $\PiLSP{}$ in first term of
definition~\eqref{eq:msh:local:def}.
We refer to the quantity
\begin{align*}
  \calMhP(\vsh,\qs) = \mshP(\vsh,\qs) - \scalP{\vsh}{\qs}
\end{align*}
as the \emph{local consistency discrepancy} for the bilinear form
$\mshP(\cdot,\cdot)$.
By considering all the mesh elements $\P$, we define the
\emph{global consistency discrepancy} as a function of $\vsh\in\Vsh$ and
$\qs\in\PS{1}(\Th)$:
\begin{align}
  \calMh(\vsh,\qs)
  = \sum_{\P\in\Th}\left( \mshP(\vsh,\qs) - \scalP{\vsh}{\qs} \right)
  = \msh(\vsh,\qs) - \scal{\vsh}{\qs}.
  \label{eq:global:consistency:discrepancy:def}
\end{align}

We prove a bound to control the local consistency discrepancy for an
element $\P$ in the following lemma.

\begin{lemma}[Local consistency discrepancy]
  \label{lemma:msh:local:consistency:discrepancy}
  Let $\P$ be a polygonal element satisfying mesh assumptions
  \ASSUM{M}{1}-\ASSUM{M}{2}.
  For all virtual element functions $\vsh\in\Vsh(\P)$ and polynomials
  $\qs\in\PS{1}(\P)$ it holds that
  \begin{align}
    \ABS{ \calMhP(\vsh,\qs) } = 
    \ABS{ \mshP(\vsh,\qs) - \scalP{\vsh}{\qs} }
    \leq \Cs\,\hP^2\snorm{\qs}{1,\P}\snorm{\vsh}{1,\P}.
    \label{eq:msh:local:consistency:discrepancy}
  \end{align}
\end{lemma}
\begin{proof}
  First, we note that for every linear polynomial
  definition~\eqref{eq:msh:local:def}, the polynomial invariance
  property $(i)$ of Lemma~\ref{lemma:PiLS:properties} implies that
  \begin{align*}
    \mshP(\vsh,\qs) = \scalP{\PiLSP{}\vsh}{\PiLSP{}\qs} =
    \scalP{\PiLSP{}\vsh}{\qs},
  \end{align*}
  so that the consistency discrepancy becomes
  \begin{align*}
    \mshP(\vsh,\qs) - \scalP{\vsh}{\qs}
    = \scalP{\PiLSP{}\vsh-\vsh}{\qs}.
  \end{align*}
  The assertion of the lemma follows from the approximation property
  of $\PiLSP{}$ stated in
  Lemma~\ref{lemma:approx:property:1}.
  \ENDPROOF
\end{proof}
According to Lemma~\ref{lemma:msh:local:consistency:discrepancy}, the
local consistency discrepancy for an element $\P$ is proportional to
$\mP\calO(\hP^2)$ since both $\SNORM{\qs}{1,\P}$ and
$\SNORM{\vsh}{1,\P}$ in the right-hand side
of~\eqref{eq:msh:local:consistency:discrepancy} are proportional to
$\mP^{\half}$.
In view of Lemma~\ref{lemma:msh:local:consistency:discrepancy}, we can
also prove an upper bound for the global consistency discrepancy
defined in~\eqref{eq:global:consistency:discrepancy:def}, which will
be useful in the analysis of Section~\ref{sec:convergence:analysis}.

\RED{
  \begin{remark}
    In Lemma~\eqref{lemma:msh:local:consistency:discrepancy}, we
    proved an upper bound for the local consistency discrepancy of the
    term $\mathcal{M}_h(\vsh,\qs)$, where $\vsh$ is a virtual function
    and $\qs$ is a polynomial.
    Note that $\vsh$ is globally $\HS{1}$ regular and locally a
    harmonic function on polygonal elements including non-convex
    elements.
    Hence, the minimal regularity for the virtual element function is
    $\HS{\frac{3}{2}-\epsilon}(P)$ for all $\epsilon >0$.
    However, the derivation of the error estimate only assumes that
    $\vsh\in\HS{1}(\P)$.
  \end{remark}
}

\begin{lemma}[Global consistency discrepancy]
  \label{lemma:msh:approx:consistency:global}
  Let $\calMh:\Vsh\times\PS{1}(\Th)\to\REAL$ be the bilinear form
  defined in~\eqref{eq:global:consistency:discrepancy:def}.
  Then, for all $\vsh\in\Vsh$ and $\qs\in\PS{1}(\Th)$ it holds that
  \begin{align}
    \ABS{\calMh(\vsh,\qs)} \leq \Cs\hh^2\snorm{\qs}{1,\hh}\snorm{\vsh}{1},
  \end{align}
  using the broken Sobolev seminorm
  $\snorm{\qs}{1,\hh}^2=\sum_{\P\in\Th}\snorm{\qs}{1,\P}^2$.
\end{lemma}
\begin{proof}
  This lemma is an immediate consequence of
  Lemma~\ref{lemma:msh:local:consistency:discrepancy}.
  Indeed, we sum all the local inequalites of the right-hand side of
  \eqref{eq:msh:local:consistency:discrepancy}; then, we note that
  $\hP\leq\hh$ for all $\P\in\Th$, and use the Cauchy-Schwarz
  inequality and the definition of the seminorms
  $\snorm{\cdot}{1,\hh}$ and $\snorm{\cdot}{1}$:
   \begin{align*}
     \ABS{\calMh(\vsh,\qs)}
     &\leq \sum_{\P\in\Th}\ABS{ \mshP(\vsh,\qs) - \scalP{\vsh}{\qs} }
     \leq \Cs\sum_{\P\in\Th}\hP^2\snorm{\qs}{1,\P}\snorm{\vsh}{1,\P}
     = \Cs\hh^2\snorm{\qs}{1,\hh}\snorm{\vsh}{1}.
  \end{align*}
   \ENDPROOF
\end{proof}
Using the same argument as for the local case (see the comment after
Lemma~\ref{lemma:msh:local:consistency:discrepancy}), we see that the
consistency discrepancy is roughly proportional to
$\ABS{\Omega}\calO(\hh^2)$.

\subsection{The virtual element bilinear form $\ash(\cdot,\cdot)$}
We assume that the bilinear form $\ash(\cdot,\cdot)$ is given by the
sum of elemental contributions
\begin{align}
  \ash(\ush,\vsh)
  =\sum_{\P\in\Th}\ashP(\ush,\vsh),
  \label{eq:ash:global:def}
\end{align}
where we define each local term as
\begin{align}
  \ashP(\ush,\vsh) = \asP (\PinP{}\ush,\PinP{}\vsh) +
  \sPa\Big((I-\PinP{})\ush,(I-\PinP{})\vsh\Big).
  \label{eq:ash:local:def}
\end{align}
In~\eqref{eq:ash:local:def}, the bilinear form
$\sPa(\cdot,\cdot):\Vsh(\P)\times\Vsh(\P)\to\REAL$ can be \emph{any}
computable, symmetric and positive definite bilinear form such that
\begin{align}
  \gamma_*\asP(\vsh,\vsh)
  \leq\sPa(\vsh,\vsh)\leq
  \gamma^*\asP(\vsh,\vsh)
  \quad\forall\vsh\in\Vsh(\P)\cap\KER\big(\PinP{}\big),
  \label{eq:requirement:Sa}
\end{align}
where
$\KER\big(\PinP{}\big)=\big\{\vs\in\HS{1}(\P):\PinP{}\vs=0~\textrm{in}~\P\big\}$ is
the kernel of the projection operator $\PinP{}$ and $\gamma_*$ and
$\gamma^*$ are two real, positive constants independent of $\hh$ (and
$\P$).

\RED{
\begin{remark}
  From \eqref{eq:requirement:Sm} and \eqref{eq:requirement:Sa}, we
  deduce that the stabilizers $S_m^E(\cdot,\cdot)$ and
  $S_a^E(\cdot,\cdot)$ are spectraly equivalent to the bilinear
  forms $(\cdot,\cdot)_{\P}$ and $\asP(\cdot,\cdot)$, respectively.
  In other words, $S_m^{\P}(\cdot,\cdot)$ and
  $S_a^{\P}(\cdot,\cdot)$ must scale as $(\cdot,\cdot)_E$ and
  $\as^P(\cdot,\cdot)$.
  Accordingly, we considered the following choice of stabilizers
  \begin{align*}
    \sPm(\Xi_i,\Xi_j) =\mP\sum_{z=1}^{N^{\text{dof}}} \DOFS{\Xi_i}{z}\,\DOFS{\Xi_j}{z}, \\
    \sPa(\Xi_i,\Xi_j) =   \sum_{z=1}^{N^{\text{dof}}} \DOFS{\Xi_i}{z}\,\DOFS{\Xi_j}{z},
  \end{align*}
  where $\{\Xi_{i}\}$ is the $i$-th canonical basis function of the
  virtual element space $\Vsh(\P)$ and function $\DOFS{z}(\cdot)$
  returns the $z$-th degree of freedom of its argument.
\end{remark}
}

\medskip
The discrete bilinear form $\ashP(\cdot,\cdot)$ satisfies the
following properties:

\begin{itemize}
\item[-] {\emph{stability}}: for every virtual element function
  $\vsh\in\Vsh(\P)$, the following stability inequality holds
  \begin{align}
    \alpha_*\asP(\vsh,\vsh)
    &\leq\ashP(\vsh,\vsh)
    \leq\alpha^*\asP(\vsh,\vsh)\quad\forall\vsh\in\Vsh(\P),
    \label{eq:ash:stability}
  \end{align}
  where we can set $\alpha^*=\max(1,\gamma^*)$ and
  $\alpha_*=\min(1,\gamma_*)$,
  cf.~\cite{BeiraodaVeiga-Brezzi-Cangiani-Manzini-Marini-Russo:2013};

  \medskip
  \item[-] {\emph{linear consistency}}: for all $\vsh\in\Vsh(\P)$ and
  $\qs\in\PS{1}(\P)$ it holds that
  \begin{align}
    \ashP(\vsh,\qs) = \asP(\vsh,\qs).
    \label{eq:ash:consistency}
  \end{align}
\end{itemize}
The constants $\alpha_*$ and $\alpha^*$ in~\eqref{eq:ash:stability}
are independent of $\hh$.
The linear consistency is an immediate consequence of the fact that
the stabilization term $\sPa(\cdot,\cdot)$ in~\eqref{eq:ash:local:def}
is zero if one of its entries $\ush$ or $\vsh$ is a linear polynomial
and $\PinP{}$ is the orthogonal projection with respect to the (semi)
inner product in $\HS{1}(\P)$.

\subsection{Right-hand side functional}
\label{subsec:RHS}
In this section, we omit to indicate the explicit dependence on $t$ in
$\fs(t)$ and the corresponding approximation $\fsh(t)$ to simplify the
notation.
To approximate in space the right-hand side of~\eqref{eq:parb:ineq:A},
we first split the linear functional $\scal{\fs}{\cdot}$ in the
summation of elemental terms $\scalP{\fs}{\cdot}$.
Then, we approximate every elemental term $\scalP{\fs}{\cdot}$ by the
virtual element linear functional $\scalP{\fsh}{\cdot}$, so that
\begin{align}
  \scal{\fsh}{\vsh} = \sum_{\P\in\Th}\scalP{\fsh}{\vsh}
  \quad\forall\vsh\in\Vsh.
  \label{eq:RHS:global:def}
\end{align}
The local linear functional $\scalP{\fsh}{\vsh}$ is defined as follows
\begin{align}
  \scalP{\fsh}{\vsh} = \scalP{\fs}{\PiLSP{}\vsh},
  \label{eq:RHS:local:def}
\end{align}
i.e., by taking $\PiLSP{}\vsh$ instead of $\vsh$ in every polygonal
cell $\P$.
The integral in the right-hand side is clearly computable since the
projection $\PiLSP{}\vsh$ is computable from the degrees of freedom of
$\vsh$.
We characterize this approximation in the following lemma.
\begin{lemma}
  \label{lemma:fs:approximation}
  Let $\fs\in H^1(\Omega)$.
  Under assumptions \ASSUM{M}{1}-\ASSUM{M}{2}, there exists a real,
  positive constant $\Cs$ independent of $\hh$ such that for every
  $\vsh\in\Vsh$ it holds that
  \begin{align}
    \ABS{ \scal{\fs-\fsh}{\vsh} }\leq \Cs\hh^2 \snorm{f}{1} \snorm{\vsh}{1}.
    \label{eq:fs:approximation}
  \end{align}
\end{lemma}
\begin{proof}
  Let $\P$ be a mesh element.
  Starting from~\eqref{eq:RHS:local:def}, we add and subtract the
  polynomial approximation $\fs_{\pi}$, use the Cauchy-Schwarz
  inequality, and the result of Lemmas~\ref{lemma:projection}
  and~\ref{lemma:approx:property:1}, to obtain the inequality chain
  \begin{align*}
    \ScalP{\fs-\fsh}{\vsh}
    &= \ScalP{\fs}{\vsh-\PiLSP{}\vsh}
   = \ScalP{\fs-\fs_{\pi}}{\vsh-\PiLSP{}\vsh} + \ScalP{\fs_{\pi}}{\vsh-\PiLSP{}\vsh}
   \\[0.5em]
    &\leq
    \NORM{\fs-\fs_{\pi}}{0,\P}\,\Norm{\vsh-\PiLSP{}\vsh}{0,\P} +
    \Cs\hP^2\snorm{\fs_{\pi}}{1,\P}\,\snorm{\vsh}{1,\P}
    \leq \Cs\hP^2\snorm{\fs}{1,\P}\,\snorm{\vsh}{1,\P}.
  \end{align*}
  The assertion of the lemma follows by adding all elemental
  inequalities and using again the Cauchy-Schwarz inequality.
\end{proof}


\subsection{Well-posedness}

Well-posedness is established in Theorem~\ref{theorem:wellposedness},
which is the major result of this subsection.
This theorem states the existence and uniqueness of the solution
$\Ussh{n+1}$ of the fully discrete scheme~\eqref{eq:VEM:A} at every
time iteration $n$.  
We provide two distinct proofs of this theorem to highlight two
different aspects of the virtual element method.
The first proof is based on an application of the Contraction
Mapping Theorem after a reformulation of the
parabolic variational inequality problem as the fixed point problem of
a contractive mapping.
This property makes it possible to implement the method through inner iterations
that are performed at every time step to update the solution in time.
The second proof is based on the minimization of a quadratic
functional on the convex set $\calKh$.
We propose this alternative proof because it extends the argument that
is briefly mentioned in~\cite{Johnson:1976} to the virtual element
setting and provides an hint for a practical algorithmic
implementation based on solving a minimization problem at any time
iteration.
\begin{theorem}[Well-posedness]
  \label{theorem:wellposedness}
  Let $\Ussh{0}$ be the initial solution at time $t=0$
  satisfying~\eqref{eq:VEM:B}.
  Then, at every time step $\ts^{n+1}$ for $0\leq\ns\leq\Ns-1$,
  the solution $\Ussh{n+1}$ of the 
  fully discrete scheme~\eqref{eq:VEM:A} 
  exists and is unique.
\end{theorem}
\noindent
\textbf{Proof 1 - Well-posedness using the Contraction Mapping Theorem.}
We rewrite the parabolic inequality~\eqref{eq:VEM:A} in the
following form
\begin{align}
  \msh(\Ussh{n+1},\Ussh{n+1}-\vsh)
  + \Delta\ts\ash(\Ussh{n+1},\Ussh{n+1}-\vsh)
  \leq
  \Delta\ts\scal{\fssh{n+1}}{\Ussh{n+1}-\vsh}
  + \msh(\Ussh{n},\Ussh{n+1}-\vsh).
  \label{eq:wellposed:1}
\end{align}
Then, we introduce the bilinear form
\begin{align}
  \Ash\big(\Ussh{n+1},\vsh\big)
  = \msh(\Ussh{n+1},\vsh) + \Delta\ts\ash(\Ussh{n+1},\vsh).
  \label{eq:wellposed:2}
\end{align}
For any fixed $\Ussh{n+1}\in\Vsh$, the mapping
$\vsh\mapsto\Ash\big(\Ussh{n+1},\vsh\big)$ is linear and bounded from
above in view of the right inequalities in~\eqref{eq:msh:stability}
and~\eqref{eq:ash:stability}, and thus, belongs to the dual space
$(\Vsh)^\prime$.
Therefore, we can find an operator $\calB:\Vsh\to(\Vsh)^\prime$ such
that $\calB\Ussh{n+1}$ is the Riesz representative of
$\Ash\big(\Ussh{n+1},\cdot\big)$ in $\Vsh$.
Formally, we can write that
$\Scal{\calB\Ussh{n+1}}{\vsh}=\Ash\big(\Ussh{n+1},\vsh\big)$ for all
$\vsh\in\Vsh$.
The stability conditions~\eqref{eq:msh:stability}
and~\eqref{eq:ash:stability} implies that $\calB$ is also bounded from
below and from above.
So, there exists two real positive constants $\Cs_*$ and $\Cs^*$ such
that
\begin{align*}
  \Cs_*\NORM{\vsh}{0}^2\leq\Scal{\calB\vsh}{\vsh}\leq\Cs^*\NORM{\vsh}{0}^2
  \quad\forall\vsh\in\Vsh.
\end{align*}
The constants $\Cs_*$ and $\Cs^*$ only depend on $\mu^*$, $\mu_*$,
$\alpha^*$, $\alpha_*$, and $\Delta\ts$, but are independent of $\hh$.
Analogously, there exists an element $\bsh\in\Vsh$ such that
\begin{align}
  \scal{\bsh}{ \Ussh{n+1}-\vsh } =
  \Delta\ts\scal{\fssh{n+1}}{\Ussh{n+1}-\vsh} + \msh(\Ussh{n},\Ussh{n+1}-\vsh).
  \label{eq:wellposed:3}
\end{align}
Using \eqref{eq:wellposed:2} and \eqref{eq:wellposed:3}, we
reformulate the parabolic variational inequality~\eqref{eq:VEM:A} as:
\begin{align}
  \emph{Find $\Ussh{n+1}\in\calKh$ such that:}\quad
  \Scal{\calB\Ussh{n+1}}{\Ussh{n+1}-\vsh} \leq \scal{\bsh}{ \Ussh{n+1}-\vsh }
  \qquad\forall\vsh\in\calKh.
  \label{eq:wellposed:new:form}
\end{align}
We add and subtract $\Ussh{n+1}$ and introduce a real factor $\beta>0$
such that
\begin{align}
  \Scal{ \beta\big(\bsh-\calB\Ussh{n+1}\big) + \Ussh{n+1} - \Ussh{n+1}   }{\vsh-\Ussh{n+1}}
  \leq 0.
  \label{eq:wellposed:beta}
\end{align}
Let $\calP:\Vsh\to\calKh$ be the projection operator on $\calKh$ such
that for any $\omega\in\Vsh$, $\calP\omega$ is the solution to the
variational inequality
\begin{align}
  \Scal{\omega-\calP\omega}{\vsh-\calP\omega} \leq 0
  \quad\forall\vsh\in\calKh.
  \label{eq:wellposed:projection}
\end{align}
By comparing \eqref{eq:wellposed:beta}
and~\eqref{eq:wellposed:projection}, we identify
$\omega=\beta\big(\bsh-\calB\Ussh{n+1}\big) + \Ussh{n+1}$ and
$\calP\omega=\Ussh{n+1}$.
Therefore, solving \eqref{eq:wellposed:new:form} is equivalent to
solving the nonlinear problem:
\begin{align*}
  \emph{Find $\Ussh{n+1}\in\calKh$ such that:}~~
  \Ussh{n+1} = \calP\big( \beta\bsh - \beta\calB\Ussh{n+1} + \Ussh{n+1} \big).
\end{align*}
Now, we introduce the affine mapping
\begin{equation}
  \calG_{\beta}(\vs) = \calP\big( \beta\bsh - \beta\calB\vs + \vs ).
  \label{eq:wellposed:nonlinear}
\end{equation}
Using~\eqref{eq:wellposed:nonlinear}, we finally reformulate the
parabolic variational inequality problem as the fixed point problem:
\begin{align}
  \emph{Find $\Ussh{n+1}\in\calKh$ such that:~}\calG_{\beta}(\Ussh{n+1}) = \Ussh{n+1}.
  \label{eq:wellposed:equivalent:pblm}
\end{align}
The fixed point exists and is unique in view of the Contraction
Mapping Theorem since $\calG_{\beta}(\cdot)$ is a contractive mapping.
To prove this statement, we consider two arbitrary functions
$\vsp,\vspp\in\calKh$.
Then, from definition \eqref{eq:wellposed:nonlinear} and noting that
$\calP$ is bounded, we obtain
\begin{align*}
  \NORM{ \calG_{\beta}(\vsp) - \calG_{\beta}(\vspp) }{0} 
  &\leq \NORM{ \big(\beta\bsh - \beta\calB\vsp + \vsp\big) - \big(\beta\bsh - \beta\calB\vspp + \vspp\big) }{0}
  \\[0.5em]
  &\leq \NORM{ \beta\calB(\vspp-\vsp) - (\vspp-\vsp\big) }{0}.
\end{align*}
A straightforward calculation using the boundedness of operator
$\calB$ yields:
\begin{align*}
  \NORM{ \calG_{\beta}(\vsp) - \calG_{\beta}(\vspp) }{0}^2
  &\leq
  \beta^2\NORM{ \calB(\vspp-\vsp) }{0}^2
  + \NORM{ \vspp-\vsp }{0}^2
  - 2\beta\scal{\calB(\vspp-\vsp)}{\vspp-\vsp}
  \\[0.5em]
  &\leq
  \big(1+\beta^2 (\Cs^{*})^2 - 2\beta\Cs_{*}\big)\NORM{\vspp-\vsp}{0}^2.
\end{align*}
Mapping $\calG_{\beta}(\cdot)$ is a contraction by setting
$\big(1+\beta^2(\Cs^{*})^2-2\beta\Cs_{*}\big)<1$, i.e., by choosing
$\beta\in\big(0,\frac{2 C_{*}}{(C^{*})^2}\big)$.
The application of the Contraction Mapping Theorem immediately imply
that $\calG_{\beta}(\cdot)$ has precisely one fixed point, which is
the solution of \eqref{eq:wellposed:equivalent:pblm}.
This fixed point is $\Ussh{n+1}$, the virtual element solution at time
$\ts^{n+1}$.
On iterating this argument at every time step shows that problem
\eqref{eq:VEM:A} is well-posed.
\ENDPROOF

\medskip
\noindent
\textbf{Proof 2 - Well-posedness using the minimization theory.}
First, we rewrite inequality \eqref{eq:VEM:A} as
\begin{align}
  \msh(\Ussh{n+1},\vsh-\Ussh{n+1}) + \Delta\ts\ash(\Ussh{n+1},\vsh-\Ussh{n+1})
  \geq \msh(\Ussh{n},\vsh-\Ussh{n+1}) + \Delta\ts~(\fsh,\vsh-\Ussh{n+1}),
  \label{eq:wellposed:min:prob}
\end{align}
and introduce the same bilinear form
\begin{align*}
  \Ash( \Ussh{n+1},\vsh) = \msh(\Ussh{n+1},\vsh) + \Delta\ts\ash(\Ussh{n+1},\vsh)
\end{align*}
of the previous proof.
From the left inequalities in~\eqref{eq:msh:stability}
and~\eqref{eq:ash:stability}, it follows that $\Ash(\cdot,\cdot)$ is
coercive on $\calKh$, i.e.,
$\alpha\NORM{\vsh}{1}^2\leq\Ash(\vsh,\vsh)$ for all $\vsh\in\calKh$
with $\alpha=\min(\mu_*,\Delta\ts\alpha_*)$ (we recall that
$\calKh=\Vsh\cap\calK\subset\HSzr{1}(\Omega)$).
As in the previous proof, the continuity of $\msh(\cdot,\cdot)$ and
the Cauchy-Schwarz inequality imply that the right-hand side
of~\eqref{eq:wellposed:min:prob} is a linear continuous functional on
$\Vsh$ for any fixed $\Ussh{n}$ and $\fssh{n+1}$.
Therefore, by the Ritz Representation Theorem, there exists an element
$\bsh\in\Vsh$ such that
\begin{align*}
  \scal{\bsh}{\zsh} = \msh(\Ussh{n},\zsh) + \Delta\ts\scal{\fssh{n+1}}{\zsh}
  \quad\forall\zsh\in\Vsh.
\end{align*}
Then, we introduce the quadratic functional
\begin{align*}
  \calI(\zsh)
  = \msh(\zsh,\zsh) + \Delta\ts\ash(\zsh,\zsh) - 2\scal{\bsh}{\zsh}
  = \Ash(\zsh,\zsh) - 2\scal{\bsh}{\zsh},
\end{align*}
and we set $\ds=\underset{\zsh\in\calKh}{\inf}\calI(\zsh)$.
Using the coercivity of $\Ash(\cdot,\cdot)$, the Cauchy-Schwarz
inequality, noting that $\NORM{\cdot}{0}\leq\NORM{\cdot}{1}$, and
finally using the Young inequality with the real parameter $\alpha$,
we find that
\begin{align*}
  \calI(\zsh)
  &
  \geq  \alpha  \NORM{\zsh}{1}^2 - 2\NORM{\bsh}{0}\NORM{\zsh}{0}
  \geq  \alpha  \NORM{\zsh}{1}^2 - 2\NORM{\bsh}{1}\NORM{\zsh}{1}
  \geq  \alpha  \NORM{\zsh}{1}^2 - (1/\alpha) \NORM{\bsh}{1}^2 - \alpha \NORM{\zsh}{1}^2
  \\[0.5em]
  &= -(1/\alpha) \NORM{\bsh}{1}^2,
\end{align*}
which implies that $d \geq -(1/\alpha)\NORM{\bsh}{1}^2 > -\infty$.
Let $n$ denote a positive integer number.
Since $d=\underset{\zsh\in\calKh}{\inf}\calI(\zsh)$, for each $1/n$ we
can find a virtual element function $\zs_n\in\calKh$ such that
\begin{align}
  d \leq \calI(\zs_n) < d + 1/n.
  \label{eq:minimizing:sequence}
\end{align}
We denote the sequence of virtual element functions $\zs_n\in\calKh$
that satisfies~\eqref{eq:minimizing:sequence} for $n\to\infty$ by
$\{\zs_{n}\}_{n\geq1}$.
Now, we consider $m,n\in\mathbb{N}$, we use again the coercivity of
$\Ash(\cdot,\cdot)$
and the identity 
$4\scal{\bsh}{\zs_n}+4\scal{\bsh}{\zs_m}-8\scal{\bsh}{(\zs_n+\zs_m)/2} = 0$,
to find that
\begin{align*}
  \alpha \NORM{ \zs_n - \zs_m }{1}^2
  & \leq \Ash ( \zs_n - \zs_m, \zs_n - \zs_m )                                                          \nonumber\\[0.5em]
  & =  2 \Ash ( \zs_n,  \zs_n ) + 2\Ash ( \zs_m, \zs_m ) -4 \Ash\big( (\zs_n+\zs_m)/2, (\zs_n+\zs_m)/2 ) \nonumber\\[0.5em]
  & =  2 \calI( \zs_n )         + 2\calI( \zs_m )        -4 \calI\big( (\zs_n+\zs_m)/2 \big)            \nonumber\\[0.5em]
  & \leq 2(1/n+1/m),
\end{align*}
which implies that the minimizing sequence $\{\zs_n\}_{n\geq1}$ is a
Cauchy sequence.
Since all $\zs_n\in\calKh$ and $\calKh$ is a closed subset of $\Vsh$,
then $\calKh$ contains the limit point of $\zs_n$ for $n\to\infty.$
We denote such limit point by $\zs$, so, formally it holds that
$\zs_n\rightarrow\zs$ as $n\rightarrow\infty$.
Moreover, it holds that $\calI(\zs_n)\rightarrow\calI(\zs)$ and
condition~\eqref{eq:minimizing:sequence} implies that $\calI(\zs)=d$.
For any $\vsh\in\calKh$ and real number $\epsilon\in[0,1]$, the vertex
values of the convex combination $\zs+\epsilon(\vsh-\zs)$ must be
nonnegative, so this function also belongs to $\in\calKh$ and $\calKh$
is a convex set.
Since $\calI(\cdot)$ attains its minimum at $\zs$, it also holds that
$\calI(\zs+\epsilon(\vsh-\zs))\geq\calI(\zs)$ for any
$0\leq\epsilon\leq1$, or equivalently, that 
\begin{align}
  \frac{d}{d \epsilon} \restrict{\calI\big(\zs+\epsilon(\vsh-\zs)\big)}{\epsilon=0} \geq 0.
\end{align}
From a direct calculation, we obtain that
\begin{align*}
  \frac{d}{d \epsilon}\calI\big(\zs+\epsilon(\vsh-\zs)\big)
  = 2\epsilon\Ash(\zs,\vsh-\zs)
  +\epsilon^2\Ash(\vsh-\zs,\vsh-\zs)
  -2\epsilon\scal{\bsh}{\vsh-\zs}.
\end{align*}
Rearranging the terms and dividing by $\epsilon$ yield
\begin{align}
  \Ash(\zs,\vsh-\zs)
  \geq \scal{\bsh}{\vsh-\zs} - (\epsilon/2)\Ash(\vsh-\zs,\vsh-\zs).
  \label{eq:wellposed:min:functional}
\end{align}
Finally, we set $\epsilon=0$ in \eqref{eq:wellposed:min:functional}
and we find that $\zs$ is the solution
of~\eqref{eq:wellposed:min:prob}, and, hence, of~\eqref{eq:VEM:A} if
we identify $\Ussh{n+1}=\zs$.
\ENDPROOF
 


\section{Technical lemmas}
\label{sec:technical:lemmas}

In subsection~\ref{subsec:preliminary:lemmas}, we prove some technical
lemmas, while
in subsection~\ref{subsec:nonnegative:quasi-interpolation:operator},
we discuss the construction of the nonnegative quasi-interpolation
operator for $\HS{1}$-regular functions.
The results of these lemmas are used in the convergence analysis of
Section~\ref{sec:convergence:analysis}.

\subsection{Preliminary technical lemmas}
\label{subsec:preliminary:lemmas}

\begin{lemma}[1 - Summation by parts]
  \label{lemma:summation-by-parts:1}
  Let $\big(\Xs,\scalX{\cdot}{\cdot}\big)$ be an inner product space
  on $\REAL$, and $\{\qs^{n}\}_{n}$ a finite ordered sequence of
  elements of $\Xs$ labeled by the integer index $n=0,\ldots,\Ns$ for
  some positive integer $\Ns$.
  Then, for every $\Ms=1,\ldots,\Ns$, the following identity holds:
  \begin{align}
    2 \sum_{n=0}^{\Ms-1}\scalX{ \qss{n+1}-\qss{n} }{\qss{n+1}        }
    = \sum_{n=0}^{\Ms-1}\scalX{ \qss{n+1}-\qss{n} }{\qss{n+1}-\qss{n}}
    + \scalX{ \qss{\Ms} }{ \qss{\Ms} } - \scalX{ \qss{0} }{ \qss{0} }.
    \label{eq:summation-by-parts:1}
  \end{align}
\end{lemma}
\begin{proof}
  First, we note that
  \begin{align*}
    & \scalX{ \qss{n+1}-\qss{n}}{ \qss{n+1}-\qss{n} }
    = \scalX{ \qss{n+1}        }{ \qss{n+1}         }
    + \scalX{ \qss{n} }{ \qss{n}   }
    -2\scalX{ \qss{n} }{ \qss{n+1} }
    \nonumber\\[0.5em]
    &\qquad
    =2\scalX{ \qss{n+1} }{ \qss{n+1} }
    - \scalX{ \qss{n+1} }{ \qss{n+1} }
    + \scalX{ \qss{n}   }{ \qss{n}   }
    -2\scalX{ \qss{n}   }{ \qss{n+1} }
    \nonumber\\[0.5em]
    &\qquad
    =2\scalX{ \qss{n+1}-\qss{n} }{ \qss{n+1} }
    - \scalX{ \qss{n+1}         }{ \qss{n+1} }
    + \scalX{ \qss{n}           }{ \qss{n}   }.
  \end{align*}
  Then, we add all the terms for $n=0,\ldots,\Ms-1$ to obtain:
  \begin{align*}
    \sum_{n=0}^{\Ms-1}\scalX{ \qss{n+1}-\qss{n} }{ \qss{n+1}-\qss{n} }
    = 2\sum_{n=0}^{\Ms-1}\scalX{ \qss{n+1}-\qss{n} }{ \qss{n+1} }
    -  \sum_{n=0}^{\Ms-1}\Big( \scalX{ \qss{n+1} }{ \qss{n+1} }
    - \scalX{ \qss{n} }{ \qss{n} } \Big).
  \end{align*}
  The assertion of the lemma follows by noting that the second term on
  the right is the telescopic sum
  \begin{align*}
    \sum_{n=0}^{\Ms-1}\Big( \scalX{ \qss{n+1} }{ \qss{n+1} }
    - \scalX{ \qss{n}   }{ \qss{n}   } \Big)
    = \scalX{ \qss{\Ms} }{ \qss{\Ms} }
    - \scalX{ \qss{0}   }{ \qss{0}   },
  \end{align*}
  and rearranging the terms of the resulting identity.
  \ENDPROOF
\end{proof}

\begin{lemma}[2 - Summation by parts]
  \label{lemma:summation-by-parts:2}
  Let $\big(\Xs,\scalX{\cdot}{\cdot}\big)$ be an inner product space
  on $\REAL$, and $\{\qs^{n}\}_{n}$ and $\{\ps^{n}\}_{n}$ be two
  finite ordered sequences of elements of $\Xs$ labeled by the integer
  index $n=0,\ldots,\Ns$ for some positive integer $\Ns$.
  Then, it holds that
  \begin{align}
    \label{eq:summation-by-parts:2}
    \sum_{n=0}^{\Ns-1}\scalX{\qs^{n+1}-\qs^{n}}{\ps^{n+1}}
    = -\sum_{n=0}^{\Ns-1}\scalX{\qs^{n}}{\ps^{n+1}-\ps^{n}}
    + \scalX{\qs^{N}}{\ps^{N}} - \scalX{\qs^{0}}{\ps^{0}}.
  \end{align}
\end{lemma}
\begin{proof}
  We note that
  \begin{align*}
    \scalX{\qs^{n+1}-\qs^{n}}{\ps^{n+1}} + \scalX{\qs^{n}}{\ps^{n+1}-\ps^{n}}
    = \scalX{\qs^{n+1}}{\ps^{n+1}} - \scalX{\qs^{n}}{\ps^{n}}.
  \end{align*}
  Then, we add both sides for $n=0$ to $\Ns-1$ and note that the
  right-hand side is a telescopic sum.
\end{proof}

\begin{lemma}
  \label{lemma:infty:bound}
  Let $\big(\Xs,\NORM{\,\cdot\,}{\Xs}\big)$ be a normed space on
  $\REAL$.
  Consider a function $\qs\in\LS{\infty}(0,T;\Xs)\cap\CS{}(0,T;\Xs)$
  and a finite ordered sequence $\{\qs^n\}_{n}$ of discrete values of
  $\qs(t)\in\Xs$ taken at successive instants $\ts^{n}$,
  $n=0,\ldots,\Ns$, e.g., $\qs^n=\qs(\ts^n)$.
  Let $\Delta\ts^n=\ts^{n+1}-\ts^{n}$ be the size of the $(n+1)$-th
  time interval $\big[\ts^{n},\ts^{n+1}\big]$ and note that
  $\Ts=\sum_{n=0}^{\Ns-1}\Delta\ts^{n}$.
  Then, it holds that
  \begin{align*}
    \sum_{n=0}^{\Ns-1}\Delta\ts^n\NORM{\qs^n}{\Xs} \leq \Ts\NORM{\qs}{\LS{\infty}(0,T;\Xs)}.
  \end{align*}
\end{lemma}
\begin{proof}
  The following chain of inequalities holds
  \begin{align*}
    \sum_{n=0}^{\Ns-1}\Delta\ts^{n}\NORM{\qs^n}{\Xs}
    \leq \max_{0\leq\ns\leq\Ns-1}\NORM{\qs^n}{\Xs}\sum_{n=0}^{\Ns-1}\Delta\ts^{n}
    \leq \Ts\text{ess~sup}_{\ts\in[0,\Ts]}\NORM{\qs(\ts)}{\Xs}
    =    \Ts\NORM{\qs}{\LS{\infty}(0,T;\Xs)},
  \end{align*}
  which is the assertion of the lemma.
  \ENDPROOF
\end{proof}

\subsection{Nonnegative quasi-interpolation operator}
\label{subsec:nonnegative:quasi-interpolation:operator}
According to~\eqref{eq:regularity:B}, the time derivative of the exact
solution $\partial\us\slash{\partial\ts}$ is only in $\HS{1}(\Omega)$,
so we cannot use the interpolation operator of
Lemma~\ref{lemma:interpolation}, which assumes the $\HS{2}$
regularity.
So, in this section we discuss the construction of a quasi-interpolant
operator for $\HS{1}$-regular functions that satisfies the condition
that the interpolation is nonnegative in $\Omega$ if the function to
be interpolated is nonnegative almost everywhere in $\Omega$.
\RED{ For this construction, we proceed as in~\cite{Johnson:1976},
  although an alternative proof is possible by following the
  guidelines depicted in Reference~\cite{Mora-Rivera-Rodriguez:2015},
  which we report in the final appendix.
For the construction of the quasi-interpolant operator, we first}
increase the regularity of the function that must be interpolated
through a smoothness operator, and, then, we apply the standard
virtual element interpolation to the smoothed function.
This strategy is detailed by the two lemmas from~\cite{Johnson:1976}
that we report below omitting the proof as it is the same and
referring the interested reader to the original publication.
The generalization to the virtual element setting is immediate since
the proof of Lemma~\ref{lemma:quasi-interpolation:A}
in~\cite{Johnson:1976} is actually independent of the way the domain
is partitioned and the same argument works for triangular and
polygonal meshes.
Then, in Lemma~\ref{lemma:quasi-interpolation:B}, we can apply the
virtual element interpolation operator of
Lemma~\ref{lemma:interpolation}.
The resulting quasi-interpolation operator has optimal approximation
property and, thanks to Lemma~\ref{lemma:VEM:nonnegative-subset}, has
the desidered nonnegativity property.
We slightly modified the statement of the lemmas to adapt them to our
notation and assumptions.

\begin{lemma}
  \label{lemma:quasi-interpolation:A} 
  For every mesh $\Th$ satisfying
  assumptions~\ASSUM{M}{1}-\ASSUM{M}{2}, there is a linear operator
  $\Ssh:\HSzr{1}(\Omega)\to\HS{2}(\Omega)\cap\HSzr{1}(\Omega)$, such
  that
  \begin{itemize}
  \item[$(i)$]   $\NORM{\Ssh\vs}{k}\leq\Cs\hh^{-j}\NORM{\vs}{k-j}$;
  \item[$(ii)$]  $\NORM{\vs-\Ssh\vs}{j}\leq\Cs\hh^{k-j}\NORM{\vs}{k}$, $j=0,1$, $k=1,2$;
  \item[$(iii)$] $\Ssh\vs\geq0$ on $\Omega$ if $\vs\geq0$ a.e. on $\Omega$.
  \end{itemize}
\end{lemma}
\begin{proof}
  See~\cite[Lemma~1]{Johnson:1976}.
\end{proof}
\begin{lemma}
  \label{lemma:quasi-interpolation:B}
  Let $\Th$ be a mesh partitionings of the computational domain
  $\Omega$ satisfying assumptions~\ASSUM{M}{1}-\ASSUM{M}{2}.
  For all $\vs\in\HSzr{1}(\Omega)$, let
  $\Ih\vs=\big(\Ssh\vs\big)_{\INTP}$ be the function in $\Vsh$ that
  interpolates $\Ssh\vs$ at the vertices of $\Th$.
  Then,
  \begin{itemize}
  \item[$(i)$]  $\NORM{\vs-\Ih\vs}{j}\leq\Cs\hh^{k-j}\NORM{\vs}{k}$, $j=0,1$, $k=1,2$
  \item[$(ii)$] $\Ih\vs\in\calKh$ if $\vs\in\calK$.
  \end{itemize}
\end{lemma}
\begin{proof}
  See~\cite[Lemma~2]{Johnson:1976}.
\end{proof}

\begin{remark}[Nonnegativity of $\Ih$]
  The nonnegativity of the quasi-interpolation operator $\Ih$ follows
  from Lemma~\ref{lemma:VEM:nonnegative-subset} and $(iii)$ of
  Lemma~\ref{lemma:quasi-interpolation:A}.
  In fact, if $\vs\in\HSzr{1}(\Omega)$ is almost everywhere
  nonnegative, then the smoothed function $\Ssh\vs$ is nonnegative on
  $\Omega$, and its quasi-interpolant $\Ih\vs$ belongs to
  $\calKh$.
\end{remark}

The following lemmas provides two approximation results about the
quasi-interpolation operator that will be used in the convergence
analysis of Section~\ref{sec:convergence:analysis}.

\begin{lemma}
  \label{lemma:L2:estimate:u}
  There exists a real, positive constant $\Cs$ independent of $\hh$
  and $\Delta\ts$, such that for all
  $\vs\in\LS{2}\big(0,T;\HS{1}(\Omega)\big)$ it holds that
  \begin{align}
    \Delta\ts\sum_{n=0}^{\Ns-1}\NORM{\vs(\ts^{n+1})-\Ih\vs(\ts^{n+1})}{0}^2 \leq
    \Cs\hh^2\NORM{\vs}{\LS{2}\big(0,T;\HS{1}(\Omega)\big)}^2,
    \label{eq:L2:estimate:u}
  \end{align}
  where, for $\ts^{n+1}\in[0,T]$, the function $\Ih\vs(t^{n+1})$ is
  the quasi-interpolant of $\vs(t^{n+1})$ defined through the
  construction of
  Lemmas~\ref{lemma:quasi-interpolation:A}-\ref{lemma:quasi-interpolation:B}
  on a mesh $\Th$ satisfying assumptions~\ASSUM{M}{1}-\ASSUM{M}{2}.
\end{lemma}
\begin{proof}
  This lemma is a straightforward consequence of
  Lemma~\ref{lemma:quasi-interpolation:B} (set $j=0$, $k=1$) and the
  norm definition in the Bochner space
  $\LS{2}\big(0,T;\HS{1}(\Omega)\big)$.
\end{proof}

\begin{lemma}
  \label{lemma:L2:estimate:dudt}
  There exists a real, positive constant $\Cs$ independent of $\hh$
  and $\Delta\ts$, such that for all
  $\vs\in\LS{2}\big(0,T;\HS{1}(\Omega)\big)$ it holds that
  \begin{align}
    \Delta\ts\sum_{n=0}^{\Ns-1}\NORM{\partial\big(\vs(\ts^{n})-\Ih\vs(\ts^{n})\big)}{0}^2
    \leq \Cs\hh^2\NORM{\frac{\partial\vs}{\partial\ts}}{\LS{2}\big(0,T;\HS{1}(\Omega)\big)}^2
    \label{eq:L2:estimate:dudt}
  \end{align}
  where, for $\ts^{n+1}\in[0,T]$, the function $\Ih\vs(t^{n+1})$ is
  the quasi-interpolant of $\vs(t^{n+1})$ defined through the
  construction of
  Lemmas~\ref{lemma:quasi-interpolation:A}-\ref{lemma:quasi-interpolation:B}
  on a mesh $\Th$ satisfying assumptions~\ASSUM{M}{1}-\ASSUM{M}{2}.
\end{lemma}
\begin{proof}
  Denote $\eta(\ts)=\vs(\ts)-\Ih\vs(\ts)$.
  We use the abbreviation~$\etas{n}=\eta(\ts^n)$ and recall that
  $\partial\etas{n}=\big(\etas{n+1}-\etas{n}\big)\slash{\Delta\ts}$.
  A direct calculation shows that
  \begin{align}
    \NORM{ \partial\etas{n} }{0}^2
    &=    \int_{\Omega}\ABS{\partial\etas{n}}^2\dx\dy
    =    \int_{\Omega}\ABS{ \frac{\etas{n+1}-\etas{n} }{\Delta\ts} }^2\dx\dy
    =    \frac{1}{\Delta\ts^2}\int_{\Omega}\ABS{ \int_{\ts^n}^{\ts^{n+1}}\frac{\partial\eta}{\partial\ts} \dt}^2\dx\dy
    \nonumber\\[0.5em]
    &\leq \frac{1}{\Delta\ts}  \int_{\Omega}\int_{\ts^n}^{\ts^{n+1}}\ABS{ \frac{\partial\eta}{\partial\ts} }^2\dt
    = \frac{1}{\Delta\ts}\int_{\ts^n}^{\ts^{n+1}}\left(\int_{\Omega}\ABS{ \frac{\partial\eta}{\partial\ts} }^2\dx\dy\right)\dt
    = \frac{1}{\Delta\ts}\int_{\ts^n}^{\ts^{n+1}}\NORM{ \frac{\partial\eta}{\partial\ts} }{0}^2\dt.
    \label{eq:L2:estimate:dudt:proof:00}
  \end{align}
  Then, we note that the quasi-interpolation operator commutes with
  the derivative in time, 
  \begin{align*}
    \frac{\partial\eta}{\partial\ts}(\ts^n)
    = \frac{\partial}{\partial\ts}\big(1-\Ih\big)\restrict{\vs(\ts)}{\ts=\ts^{n}}
    = \big(1-\Ih\big)\frac{\partial\vs}{\partial\ts}(\ts^n).
  \end{align*}
  So, by using again Lemma~\ref{lemma:quasi-interpolation:B}
  with $j=0$ and $k=1$ we find that
  \begin{align*}
    \NORM{ \frac{\partial\etas{n}}{\partial\ts} }{0}
    \leq \Cs\hh \mbox{$ \NORM{ \frac{\partial\vs}{\partial\ts}(\ts^{n})}{1} $}, 
  \end{align*}
  where the constant $\Cs$ is independent of $\hh$ and $\Delta\ts$.
  Substituting this error bound
  in~\eqref{eq:L2:estimate:dudt:proof:00} and adding the resulting
  inequality over all time intervals $\big[t^{n},t^{n+1}\big]$
  conclude the proof of the lemma.
  \ENDPROOF
\end{proof}


\section{Convergence analysis}
\label{sec:convergence:analysis}
In the proof of the following theorem, we use the abbreviations
$\uss{\ell}=\us(\ts^{\ell})$, $\Ussh{\ell}=\Ush(\ts^{\ell})$,
$\uss{\ell}_{\pi}=\us_{\pi}(\ts^{\ell})$,
$\fssh{\ell}=\fsh(\ts^{\ell})$, $\fss{\ell}=\fs(\ts^{\ell})$, for
$\ell=n,n+1$,where $\ts^{\ell}\in[0,\Ts]$.
Our analysis is built on top of the convergence analysis that is
presented in~\cite{Johnson:1976} and actually confirm this result in
the framework of the virtual element method.
In the proof, we identify the terms that appear in the original paper
and the terms that are the consequence of the variational crime
determined by the virtual element approach, and we provide a estimate
for this latter ones.
Resorting to the VEM demands more regularity on the forcing term $\fs$
than in the original convergence theorem of
Reference~\cite{Johnson:1976}.
However, this fact is aligned with the virtual element setting
proposed in~\cite{Ahmad-Alsaedi-Brezzi-Marini-Russo:2013}.

\begin{theorem}
  \label{theorem:convergence}
  Let $\us$ be the analytical solution of problem
  \eqref{eq:parb:ineq:A}-\eqref{eq:parb:ineq:B} under assumptions
  \ASSUM{A}{1}-\ASSUM{A}{4}, and with a source term
  $\fs\in\LS{\infty}\big(J;\HS{1}(\Omega)\big)$.
  Let $\Ussh{n}\in\calKh\subset\Vsh$ be the solution to the virtual
  element method \eqref{eq:VEM:A}-\eqref{eq:VEM:B} with the
  construction detailed in Section~\ref{sec:VEM} under the mesh
  regularity \ASSUM{M}{1}-\ASSUM{M}{2}.
  Then, the following estimate holds:
  \begin{equation}
    \max_{1\leq\ns\leq\Ns}\norm{\uss{n}-\Ussh{n}}{0} +
    \Bigg(\sum_{n=1}^{\NT}\Delta\ts\snorm{\uss{n}-\Ussh{n}}{1}^2\Bigg)^{\frac12}
    \leq\Cs\Big( \Delta\ts^{\frac34} + \hh \Big),
    \label{eq:convergence:theorem}
  \end{equation}
  for some real, positive constant $\Cs$ independent of $\hh$ and
  $\Delta\ts$.
\end{theorem}

\begin{proof}
  Let $\etas{n}=\uss{n}-\ussI{n}$ and $\thes{n}=\Ussh{n}-\ussI{n}$, so
  that we can rewrite the approximation error as
  $\ess{n}=\uss{n}-\Ussh{n}=\etas{n}-\thes{n}$.
  We start with the identities
  \begin{subequations}
    \begin{align}
      \msh(\partial\ess{n},\ess{n+1})
      &= \msh(\partial\ess{n},\etas{n+1})
      - \msh(\partial\uss {n},\thes{n+1})
      + \msh(\partial\Ussh{n},\thes{n+1})
      \label{eq:proof:100}
      \\[0.5em]
      \ash(\ess{n},\ess{n+1})
      &= \ash(\ess{n},\etas{n+1})
      - \ash(\uss {n},\thes{n+1})
      + \ash(\Ussh{n},\thes{n+1}).
      \label{eq:proof:200}
    \end{align}
    \end{subequations}
  We set $\vs=\Ussh{n+1}$ and $\ts=\ts^{n+1}$
  in~\eqref{eq:regularity:C}.
  Recalling that $\ess{n+1}=\etas{n+1}-\thes{n+1}$, we find that
  \begin{align}
    \Scal{\frac{\partial^{+}\uss{n+1}}{\partial\ts}}{-\ess{n+1}}
    + \as(\uss{n+1}, \thes{n+1}-\etas{n+1} )
    - \scal{\fss{n+1}}{ \thes{n+1}-\etas{n+1} }
    \geq 0.
    \label{eq:proof:00}
  \end{align}
  Moreover, we set $\vs=\ussI{n+1}$ in~\eqref{eq:VEM:A} and we obtain:
  \begin{align}
    \msh(\partial\Ussh{n},\thes{n+1})
    + \ash(\Ussh{n+1},\thes{n+1})
    \leq \scal{\fsh^{n+1}}{\thes{n+1}}.
    \label{eq:proof:05}
  \end{align}
  Adding~\eqref{eq:proof:100} and~\eqref{eq:proof:200} yields:
  \begin{align}
    &\msh(\partial\ess{n},\ess{n+1}) + \ash(\ess{n},\ess{n+1})
    = \Big[ \msh(\partial\ess{n},\etas{n+1}) \Big] + \Big[ \ash(\ess{n},\etas{n+1}) \Big]
    \nonumber\\[0.5em] &\qquad
    + \Big[
      - \msh(\partial\uss {n},\thes{n+1}) - \ash(\uss {n},\thes{n+1})
      + \msh(\partial\Ussh{n},\thes{n+1}) + \ash(\Ussh{n},\thes{n+1})
      \Big]
    \nonumber\\[0.5em] &\quad
    = \TERM{q}{1}^{n+1} + \TERM{q}{2}^{n+1} + \TERM{q}{3}^{n+1}.
    \label{eq:proof:10}
  \end{align}
  The three terms $ \TERM{q}{1}^{n+1}$, $\TERM{q}{2}^{n+1}$ and
  $\TERM{q}{3}^{n+1}$ in~\eqref{eq:proof:10} are identified by the
  square brackets.
  We add and subtract $\scal{\partial\ess{n}}{\etas{n+1}}$ to
  $\TERM{q}{1}^{n+1}$ and $\as(\ess{n},\etas{n+1})$ to
  $\TERM{q}{2}^{n+1}$, so that we can rewrite the first two terms in
  the right-hand side of~\eqref{eq:proof:10} as:
  \begin{align}
    \TERM{q}{1}^{n+1}
    &
    = \scal{\partial\ess{n}}{\etas{n+1}}
    + \Big[ \msh(\partial\ess{n},\etas{n+1}) - \scal{\partial\ess{n}}{\etas{n+1}} \Big]
    = \TERM{p}{1}^{n+1} + \TERM{r}{1}^{n+1},
    \label{eq:proof:15}
    \\[0.5em]
    \TERM{q}{2}^{n+1}
    &
    = \as(\ess{n},\etas{n+1})
    + \Big[ \ash(\ess{n},\etas{n+1}) - \as(\ess{n},\etas{n+1}) \Big]
    = \TERM{p}{2}^{n+1} + \TERM{r}{2}^{n+1}.
    \label{eq:proof:20}
  \end{align}
  We use inequality~\eqref{eq:proof:00} and add the left-hand side
  of~\eqref{eq:proof:05} to $\TERM{q}{3}$ to obtain:
  \begin{align*}
    \TERM{q}{3}^{n+1}
    &\leq
    -\Big[ \msh(\partial\uss{n},\thes{n+1}) + \ash(\uss {n},\thes{n+1}) \Big]
    + \scal{\fssh{n+1}}{\thes{n+1}}
    \\[0.5em]
    &\qquad
    + \Scal{\frac{\partial^{+}\uss{n+1}}{\partial\ts}}{-\ess{n+1}}
    + \as(\uss{n+1},\thes{n+1}-\etas{n+1})
    - \scal{\fss{n+1}}{ \thes{n+1}-\etas{n+1}}.
  \end{align*}
  We transform the right-hand side by adding and subtracting
  $\scal{\partial\uss{n}}{\etas{n+1}-\ess{n+1}}$, recalling the
  identity $\thes{n+1}=\etas{n+1}-\ess{n+1}$ and rearranging the
  terms:
  \begin{align}
    \TERM{q}{3}^{n+1}
    &\leq
    \scal{\fssh{n+1}-\fss{n+1}}{\thes{n+1}}
    + \Big[ \scal{\fss{n+1}}{ \etas{n+1}}
      - \scal{\partial\uss{n}}{\etas{n+1}}
      - \as(\uss{n+1},\etas{n+1})
      \Big]
    \nonumber\\[0.5em] &\qquad
    + \bigg[
      \Scal{\frac{\partial^{+}\uss{n+1}}{\partial\ts}}{-\ess{n+1}}
      - \scal{\partial\uss{n}}{-\ess{n+1}}
      \bigg]
    + \Big[
      \as(\uss{n+1},\thes{n+1}) - \ash(\uss {n},\thes{n+1})
      \Big]
    \nonumber\\[0.5em] &\qquad
    + \Big[
      \scal{\partial\uss{n}}{\thes{n+1}}
      -\msh(\partial\uss{n},\thes{n+1})
      \Big]
    = \TERM{r}{3}^{n+1} + \TERM{p}{3}^{n+1} + \TERM{p}{4}^{n+1} + \TERM{r}{4} + \TERM{r}{5}^{n+1}.
    \label{eq:proof:25}
  \end{align}
  We substitute \eqref{eq:proof:15}, \eqref{eq:proof:20},
  and~\eqref{eq:proof:25} in~\eqref{eq:proof:10}, and by collecting
  the terms $\TERM{p}{i}^{n+1}$ and $\TERM{r}{i}^{n+1}$ in two
  distinct summations we obtain:
  \begin{align*}
    \msh(\partial\ess{n},\ess{n+1}) + \ash(\ess{n},\ess{n+1})
    \leq \sum_{j=1}^{4}\TERM{p}{j}^{n+1} + \sum_{j=1}^{5}\TERM{r}{j}^{n+1}.
  \end{align*}
  We use the left-hand inequality of stability conditions
  \eqref{eq:msh:stability} and~\eqref{eq:ash:stability} to find that 
  \begin{align}
    (\partial\ess{n},\ess{n+1}) + \as(\ess{n},\ess{n+1})
    \leq
    \widetilde{\Cs}\sum_{j=1}^{4}\TERM{p}{j}^{n+1} +
    \widetilde{\Cs}\sum_{j=1}^{5}\TERM{r}{j}^{n+1},
    \label{eq:proof:35}
  \end{align}
  with $\widetilde{\Cs}=(\min(\mu_*,\alpha_*))^{-1}$.
  Then, we multiply both sides of~\eqref{eq:proof:35} by $\Delta\ts$,
  sum from $n=0$ to $n=\Ns-1$, apply
  Lemma~\eqref{lemma:summation-by-parts:1},
  cf.~\eqref{eq:summation-by-parts:1} with $\qss{n}=\thes{n}$, to the
  left-hand side of the resulting equation, to obtain:
  \begin{align}
    \max_{1\leq\ns\leq\Ns} \mbox{$ \big(\ess{n},\ess{n}\big) $}
    + \Delta\ts\sum_{n=0}^{\NT-1} \mbox{$ \as\big(\ess{n+1},\ess{n+1}\big) $}
    \leq
    \widetilde{\Cs}\bigg( 
    \msh\big( \ess{0}, \ess{0} \big)
    + \sum_{j=1}^{4}\TERM{S}{j}
    + \sum_{j=1}^{5}\TERM{R}{j}
    \bigg),
    \label{eq:proof:main}
  \end{align}
  where
  $\TERM{S}{j}=\Delta\ts\sum_{n=0}^{N-1}\TERM{p}{j}^{n+1}$ and
  $\TERM{R}{j}=\Delta\ts\sum_{n=0}^{N-1}\TERM{r}{j}^{n+1}$.
  The five terms $\TERM{R}{j}$ are specific to the VEM setting and are
  not present in the analysis of the finite element approximation
  in~\cite{Johnson:1976}.
  These five terms are indeed due to the variational ``crime`` that we
  commit by adopting the virtual element approach.

  \medskip
  In view of Lemma~\ref{lemma:quasi-interpolation:B}, we can estimate
  the four terms $\TERM{S}{j}$ as in Reference~\cite{Johnson:1976}.
  The analysis in Reference~\cite{Johnson:1976} is based on the
  existence of a nonnegative quasi-interpolation operator with optimal
  approximation properties.
  Such an interpolation operator is needed to estimate the
  approximation error of terms like $\partial\us\slash{\partial\ts}$,
  for which we cannot guarantee a regularity better than $\HS{1}$
  (unless resorting to specific and much stronger constraints in the
  problem formulation).
  We omit the details of the derivation of the upper bounds of terms
  $\TERM{S}{j}$ as they can be found in~\cite{Johnson:1976}.
  With a few notational adjustments, as for example, introducing a
  generic factor $\epsilon$ used by the Young inequality, these
  estimates are
  \begin{align*}
    \ABS{\TERM{S}{1}}
    &\leq
    2\epsilon\left(
    \alpha\Delta\ts\sum_{n=0}^{\Ns-1}\as(\ess{n+1},\ess{n+1})
    + \NORM{\ess{\Ns}}{0}^2
    + \NORM{\ess{0}}{0}^2
    \right)
    \\[0.5em]&\quad
    + \frac{\Cs}{2\epsilon}\,\hh^2
    \Bigg(
    \NORM{\us}{\LS{\infty}\big(0,T;\HS{1}(\Omega)\big)}^2 +
    \NORM{\frac{\partial\us}{\partial\ts}}{\LS{2}\big(0,T;\HS{1}(\Omega)\big)}^2
    \Bigg),
    \\[0.5em]
    \ABS{\TERM{S}{2}}
    &\leq
    2\epsilon
    \alpha\Delta\ts\sum_{n=0}^{\Ns-1}\as(\ess{n+1},\ess{n+1})
    + \frac{\Cs}{2\epsilon}\,\hh^2
    \NORM{\us}{\LS{\infty}\big(0,T;\HS{2}(\Omega)\big)}^2.
    \\[0.5em]
    \ABS{\TERM{S}{3}}
    &\leq
    \Cs\,\hh^2\left(
    \NORM{\us}{\LS{\infty}\big(0,T;\HS{2}(\Omega)\big)}^2 +
    \NORM{\frac{\partial\us}{\partial\ts}}{\LS{2}\big(0,T;\HS{1}(\Omega)\big)}^2 +
    \NORM{\fs}{\LS{\infty}\big(0,T;\LS{2}(\Omega)\big)}^2
    \right),
    \\[0.5em]
    \ABS{\TERM{S}{4}}
    &\leq
    2\epsilon\sum_{n=0}^{\Ns-1}\as(\ess{n+1},\ess{n+1})
    + \frac{\Cs}{2\epsilon}\,\Delta\ts^2
    \Bigg(
    \NORM{\frac{\partial\us}{\partial\ts}}{ \LS{2}\big(0,T;\HS{1}(\Omega)\big) } +
    \NORM{\frac{\partial\fs}{\partial\ts}}{ \LS{2}\big(0,T;\LS{2}(\Omega)\big) }
    \Bigg)
    \\[0.5em]
    &\phantom{\leq}
    +
    \NORM{\fs}{ \LS{\infty}\big(0,\Ts;\LS{\infty}(\Omega)\big) }\,
    \Bigg(
    \frac{1}{2\epsilon}\max_{1\leq\ns\leq\Ns}\NORM{\ess{n+1}}{0}^2 +
    \frac{\Delta\ts}{2\epsilon}\sum_{n=0}^{\Ns-1}\as(\ess{n+1},\ess{n+1})
    +2\epsilon\Es
    \Bigg),
  \end{align*}
  with
  \begin{align*}
    \Es = \Delta\ts^2\left[
      \left(\sum_{n\in\Ns_1}\ms(\Gamma_n)^{1/2}\right)^2 + \Delta\ts^{-1}\ps\sum_{n\in\Ns_2}\ms(\Gamma_n)^{2/q}
      \right],
  \end{align*}
  and where $\{\Ns_1,\Ns_2\}$ is a partition of $\{0,1,\ldots,\Ns-1\}$
  in two disjoint subsets as in~\cite[Eq.~(2.16)]{Johnson:1976},
  $\ms(\Gamma_n)$ is the measure of $\Gamma_n$ defined
  in~\eqref{eq:Gamma-n:def}, and $p,q$ are any two real conjugate
  indices ($1\leq\ps,\qs\leq\infty$, $(1/p)+(1/q)=1$).
  Under assumption~\ASSUM{A}{4}, it holds that
  $\Es\leq\Cs\big(\log\Delta\ts^{-1}\big)^{1/2}\Delta\ts^{3/2}$ for
  some positive constant independent of $\hh$ and $\Delta\ts$.

  \medskip
  Note that
  $\NORM{\ess{\Ns}}{0}^2\leq\max_{1\leq\ns\leq\Ns}\NORM{\ess{\ns}}{0}^2$
  in $\TERM{S}{1}$.
  Also, note that the four terms $\TERM{S}{j}$, $j=1,\ldots,4$,
  provide an upper bound of the right-hand side of~\eqref{eq:proof:main}
  with the following structure
  \begin{align}
    \sum_{j=1}^{4}\ABS{\TERM{S}{j}}
    &\leq
    \Cs_1\,\epsilon\max_{1\leq\ns\leq\Ns}\NORM{\ess{n}}{0}^2 +
    \Cs_2\,\epsilon\Delta\ts\sum_{n=0}^{\Ns-1}\as(\ess{n+1},\ess{n+1}) +
    \Cs_3\,\epsilon\NORM{\ess{0}}{0}^2
    \nonumber\\
    &\phantom{\leq}
    +\Cs_4(\epsilon)\hh^2
    \left(
    \NORM{\us}{\LS{\infty}\big(0,T;\HS{2}(\Omega)\big)}^2 +
    \NORM{\frac{\partial\us}{\partial\ts}}{\LS{2}\big(0,T;\HS{1}(\Omega)\big)}^2 +
    \NORM{\fs}{\LS{\infty}\big(0,T;\LS{2}(\Omega)\big)}^2
    \right)
    \nonumber\\[0.5em]
    &\phantom{\leq}
    + \Cs_5(\epsilon)\Delta\ts^2
    \Bigg(
    \NORM{\frac{\partial\us}{\partial\ts}}{ \LS{2}\big(0,T;\HS{1}(\Omega)\big) } +
    \NORM{\frac{\partial\fs}{\partial\ts}}{ \LS{2}\big(0,T;\LS{2}(\Omega)\big) }
    \Bigg)
    \nonumber\\[0.5em]
    &\phantom{\leq}
    +\Cs_6(\epsilon)\NORM{\fs}{ \LS{\infty}\big(0,\Ts;\LS{\infty}(\Omega)\big) }\Es(\Delta\ts),
    \label{eq:bound:S1-S4}
  \end{align}
  where all constants $\Cs_{\ell}$, $\ell=1,\ldots,6$, are independent
  of $\hh$ and $\Delta\ts$, but for $\ell=4,5,6$ depend on
  $1\slash{\epsilon}$.
  We denoted the implicit dependence on $\Delta\ts$ in $\Es$ by
  writing this term as $\Es(\Delta\ts)$.

  \medskip
  The virtual element variational ``crime'' requires an estimate of
  the five additional terms:
  \begin{align*}
    \TERM{R}{1} &= \Delta\ts\sum_{n=0}^{\Ns-1}
    \Big[ \msh(\partial\ess{n},\etas{n+1}) - \scal{\partial\ess{n}}{\etas{n+1}} \Big]
    \\[0.5em]
    \TERM{R}{2} &= \Delta\ts\sum_{n=0}^{\Ns-1}
    \Big[ \ash(\ess{n+1},\etas{n+1}) - \as(\ess{n},\etas{n+1}) \Big]
    \\[0.5em]
    \TERM{R}{3} &= \Delta\ts\sum_{n=0}^{\Ns-1}
    \scal{\fssh{n+1}-\fss{n+1}}{\thes{n+1}}
    \\[0.5em]
    \TERM{R}{4} &= \Delta\ts\sum_{n=0}^{\Ns-1}
    \Big[ \as(\uss{n+1},\thes{n+1}) - \ash(\uss {n},\thes{n+1}) \Big]
    \\[0.5em]
    \TERM{R}{5} &= \Delta\ts\sum_{n=0}^{\Ns-1}
    \Big[ \scal{\partial\uss{n}}{\thes{n+1}} -\msh(\partial\uss{n},\thes{n+1}) \Big]
  \end{align*}
  \RED{ Since we are estimating the square of the approximations
    errors in the left-hand side of~\eqref{eq:proof:35}, we need to
    prove that all these terms scale (at least) proportionally to
    $\hh^2$ and $\Delta^{\frac{3}{2}}$ to obtain the assertion of the
    theorem.  }
  We proceed by evaluating each term separately.

  \PGRAPH{Estimate of $\TERM{R}{1}$}
  We use the summation by parts of
  Lemma~\ref{lemma:summation-by-parts:2}
  (cf. Equation~\eqref{eq:summation-by-parts:2}) to transform
  $\TERM{R}{1}$ and split it in the two subterms $\TERM{R}{11}$ and
  $\TERM{R}{12}$:
  \begin{align}
    \TERM{R}{1}
    = \TERM{R}{11} + \TERM{R}{12}
    &=
    \left[
    -\Delta\ts\sum_{n=0}^{\Ns-1}\left[
      \msh(\ess{n},\partial\etas{n+1}) -
      \scal{\ess{n}}{\partial\etas{n+1}}
      \right]
    \right] +
    \nonumber\\[0.5em]
    &\qquad
    \Big[
      \big( \msh(\ess{\Ns},\etas{\Ns})   - \msh(\ess{0},\etas{0})   \big)
      -
      \big( \scal{\ess{\Ns}}{\etas{\Ns}} - \scal{\ess{0}}{\etas{0}} \big)
      \Big].
    \label{eq:R1:def}
  \end{align}
  We use the continuity of bilinear form $\mshP(\cdot,\cdot)$
  (cf. Lemma~\ref{lemma:msh:continuity}), the stability
  condition~\eqref{eq:msh:stability}, the Cauchy-Schwarz inequality
  and the Young inequality with the real factor $\epsilon_1$ to obtain
  the inequality chain:
  \begin{align}
    \ABS{\TERM{R}{11}}
    &=   \ABS{\msh(\partial\ess{n},\etas{n+1})-\scal{\ess{n}}{\partial\etas{n+1}}}
    \leq \ABS{\msh(\partial\ess{n},\etas{n+1})}+\ABS{\scal{\ess{n}}{\partial\etas{n+1}}}
    \nonumber\\[0.5em]
    &\leq (1+\mu^*)\NORM{\ess{n}}{0}\,\NORM{\partial\etas{n}}{0}
    \leq (1+\mu^*)\Big(
    2\epsilon_1          \NORM{\ess{n}}{0}^2 +
    \frac{1}{2\epsilon_1}\NORM{\partial\etas{n}}{0}^2
    \Big).
    \label{eq:R11:def}
  \end{align}
  Note that we can write $\NORM{\ess{n}}{0}^2=\as(\ess{n},\ess{n})$.
  Using inequality~\eqref{eq:L2:estimate:dudt} from
  Lemma~\ref{lemma:L2:estimate:dudt}, we find the desired upper bound for
  $\TERM{R}{11}$
  \begin{align*}
    \ABS{\TERM{R}{11}}
    \leq (1+\mu^*)\left(
    2\epsilon_1\Delta\ts\sum_{n=0}^{\Ns-1}\as(\ess{n+1},\ess{n+1})
    + \frac{\hh^2}{2\epsilon_1}\NORM{\frac{\partial\us}{\partial\ts}}{\LS{2}\big(0,T;\HS{1}(\Omega)\big)}^2
    \right).
  \end{align*}
  Similarly, we use the continuity of the bilinear form
  $\mshP(\cdot,\cdot)$, the stability
  condition~\eqref{eq:msh:stability}, the Cauchy-Schwarz and the Young
  inequality with the real factor $\epsilon_1$ to obtain
  \begin{align*}
    \ABS{\TERM{R}{12}}
    &= \ABS{\big( \msh(\ess{\Ns},\etas{\Ns})   - \msh(\ess{0},\etas{0}) \big) -
      \big( \scal{\ess{\Ns}}{\etas{\Ns}} - \scal{\ess{0}}{\etas{0}} \big) }
    \nonumber\\[0.5em]
    &\leq
    \ABS{\msh(\ess{\Ns},\etas{\Ns})}   + \ABS{\msh(\ess{0},\etas{0})} +
    \ABS{\scal{\ess{\Ns}}{\etas{\Ns}}} + \ABS{\scal{\ess{0}}{\etas{0}}}
    \nonumber\\[0.5em]
    &\leq (1+\mu^*)\Big(
    \NORM{\ess{\Ns}}{0}\,\NORM{\etas{\Ns}}{0} +
    \NORM{\ess{0}}{0}  \,\NORM{\etas{0}}{0}
    \Big)
    \nonumber\\[0.5em]
    &\leq
    2\epsilon_1(1+\mu^*)          \Big( \NORM{\ess {\Ns}}{0}^2 + \NORM{\ess {0}}{0}^2 \Big) + 
    \frac{(1+\mu^*)}{2\epsilon_1} \Big( \NORM{\etas{\Ns}}{0}^2 + \NORM{\etas{0}}{0}^2 \Big).
  \end{align*}
  Using this inequality and the results of
  Lemmas~\ref{lemma:infty:bound} and~\ref{lemma:L2:estimate:u}, cf.
  inequality~\eqref{eq:L2:estimate:u}, we find the desired upper bound
  for $\TERM{R}{12}$
  \begin{align*}
    \ABS{\TERM{R}{12}}
    \leq
    2\epsilon_1(1+\mu^*)\Big( \NORM{\ess {\Ns}}{0}^2 + \NORM{\ess {0}}{0}^2 \Big) + 
    \frac{(1+\mu^*)}{2\epsilon_1}\hh^2\NORM{\us}{\LS{\infty}\big(0,T;\HS{1}(\Omega)\big)}^2.
  \end{align*}
  Collecting the bounds of $\TERM{R}{11}$ and $\TERM{R}{12}$ yields
  \begin{align}
    \ABS{\TERM{R}{1}}
    &\leq
    2\epsilon_1(1+\mu^*)
    \left[
      \Delta\ts\sum_{n=0}^{\Ns-1}\as(\ess{n+1},\ess{n+1}) + 
      \NORM{\ess {\Ns}}{0}^2 + \NORM{\ess {0}}{0}^2 \big)
      \right]
    \nonumber\\[0.5em]
    &+
    \frac{(1+\mu^*)}{2\epsilon_1}\hh^2
    \left[
      \NORM{\us}{\LS{\infty}\big(0,T;\HS{1}(\Omega)\big)}^2 +
      \NORM{\frac{\partial\us}{\partial\ts}}{\LS{2}\big(0,T;\HS{1}(\Omega)\big)}^2
      \right].
    \label{eq:R1:bound}
  \end{align}
  
  \PGRAPH{Estimate of $\TERM{R}{2}$}
  To estimate term $\TERM{R}{2}$, we first note that the continuity of
  the bilinear form $\ash(\cdot,\cdot)$ and the Young inequality with
  the coefficient $\epsilon_2$ allows us to write
  \begin{align}
    \ABS{
    \ash(\ess{n+1},\etas{n+1}) -
    \as (\ess{n},\etas{n+1}) }
    &\leq (1+\alpha^*)\NORM{\ess{n+1}}{1}\NORM{\etas{n+1}}{1}
    \nonumber\\[0.5em]
    &\leq (1+\alpha^*)
    \Big(
    2\epsilon_2          \NORM{\ess {n+1}}{1}^2 +
    \frac{1}{2\epsilon_2}\NORM{\etas{n+1}}{1}^2
    \Big).
  \end{align}
  We substitute this inequality in the definition of term
  $\TERM{R}{2}$, use the Young inequality with the real coefficient
  $\epsilon_2>0$ and inequality~\eqref{eq:L2:estimate:u},
  cf. Lemma~\ref{lemma:L2:estimate:u}, we find that
  \begin{align}
    \ABS{\TERM{R}{2}}
    &\leq \Delta\ts\sum_{n=0}^{\Ns-1}
    \ABS{ \ash(\ess{n+1},\etas{n+1}) - \as(\ess{n},\etas{n+1}) }
    \nonumber\\[0.25em]
    &\leq
    (1+\alpha^*)\Delta\ts\sum_{n=0}^{\Ns-1}\left(
    2\epsilon_2          \NORM{\ess {n+1}}{1}^2 +
    \frac{1}{2\epsilon_2}\NORM{\etas{n+1}}{1}^2
    \right)
    \nonumber\\[0.5em]
    &\leq
    2\epsilon_2(1+\alpha^*)       \Delta\ts\sum_{n=0}^{\Ns-1}\NORM{\ess {n+1}}{1}^2 +
    \frac{1+\alpha^*}{2\epsilon_2}\Delta\ts\sum_{n=0}^{\Ns-1}\NORM{\etas{n+1}}{1}^2
    \nonumber\\[0.5em]
    &\leq
    2\epsilon_2(1+\alpha^*)\Delta\ts\sum_{n=0}^{\Ns-1}\as(\ess{n+1},\ess{n+1})
    + \frac{1+\alpha^*}{2\epsilon_2}
    \hh^2\NORM{\us}{\LS{\infty}\big(0,T;\HS{2}(\Omega)\big)}^2.
    \label{eq:R2:bound}
  \end{align}

  To estimate the next terms, we need an upper bound for the
  $\LS{2}$-norm and the $\HS{1}$-seminorm of $\thes{n+1}$.
  We recall that $\thes{n+1}=\etas{n+1}-\ess{n+1}$ and
  $\SNORM{\ess{n+1}}{1}^2=\as(\ess{n+1},\ess{n+1})$.
  Using the estimate for the quasi-interpolation operator,
  cf. Lemma~\ref{lemma:quasi-interpolation:B}, we find that
  \begin{subequations}
  \begin{align}
    \NORM{\thes{n+1}}{0}^2
    &\leq 2\NORM{\etas{n+1}}{0}^2 + 2\NORM{\ess{n+1}}{0}^2
    \leq 2\big(\Cs\hh\SNORM{\uss{n+1}}{1}\big)^2 + 2\NORM{\ess{n+1}}{0}^2,
    \label{eq:theta:upper:bound:L2}
    \intertext{and}
    \SNORM{\thes{n+1}}{1}^2
    &\leq 2\SNORM{\etas{n+1}}{1}^2 + 2\SNORM{\ess{n+1}}{1}^2
    \leq 2\big(\Cs\hh\SNORM{\uss{n+1}}{2}\big)^2 + 2\as(\ess{n+1},\ess{n+1}).
    \label{eq:theta:upper:bound:H1}
  \end{align}
  \end{subequations}
  
  \PGRAPH{Estimate of $\TERM{R}{3}$}  
  To estimate $\TERM{R}{3}$, we use result of
  Lemma~\ref{lemma:fs:approximation},
  cf. inequality~\eqref{eq:fs:approximation}, and the Young inequality
  with the real coefficient $\epsilon_3>0$ to obtain:
  \begin{align}
    \ABS{\TERM{R}{3}}
    &\leq
    \Delta\ts\sum_{n=0}^{\Ns-1}
    \ABS{ \scal{\fssh{n+1}-\fss{n+1}}{\thes{n+1}} }
    \leq\Cs\hh^2
    \Delta\ts\sum_{n=0}^{\Ns-1}
    \snorm{\fss{n+1}}{1}\,\snorm{\thes{n+1}}{1}
    \nonumber\\[0.5em]
    &
    \leq\Cs\hh^2
    \left(
    2\epsilon_3
    \Delta\ts\sum_{n=0}^{\Ns-1}
    \snorm{\fss{n+1}}{1}^2
    +
    \frac{1}{2\epsilon_3}
    \Delta\ts\sum_{n=0}^{\Ns-1}
    \,\snorm{\thes{n+1}}{1}^2
    \right)
    \nonumber\\[0.5em]
    &= \ABS{\TERM{R}{31}} + \ABS{\TERM{R}{32}}.
  \end{align}
  We estimate $\TERM{R}{31}$ by using the result of
  Lemma~\ref{lemma:infty:bound}, so that we have
  \begin{align}
    \ABS{\TERM{R}{31}}
    =    2\epsilon_3\Cs\hs^2\Delta\ts\sum_{n=0}^{\Ns-1}\snorm{\fss{n+1}}{1}^2
    \leq 2\epsilon_3\Cs\hs^2\NORM{\fs}{ \LS{\infty}\big(0,T;\HS{1}(\Omega)\big) }^2.
  \end{align}
  We estimate $\TERM{R}{32}$ by noting that
  $\snorm{\thes{n+1}}{1}=\snorm{\uss{n+1}-\Ih\uss{n+1}}{1}\leq2\snorm{\uss{n+1}}{1}$,
  and using the result of Lemma~\ref{lemma:infty:bound}, so that we
  have
  \begin{align}
    \ABS{\TERM{R}{32}}
    &\leq
    \frac{\Cs}{2\epsilon_3}\hh^2\Delta\ts\sum_{n=0}^{\Ns-1}\snorm{\uss{n+1}}{1}^2
    \leq \frac{\Cs}{2\epsilon_3}\hh^2\NORM{\us}{\LS{\infty}\big(0,T;\HS{1}(\Omega)\big)}^2.
  \end{align}
  Using these inequalities, we derive the following upper bound for
  $\TERM{R}{3}$:
  \begin{align}
    \ABS{\TERM{R}{3}}
    \leq
    2\epsilon_3           \Cs\hh^2\NORM{\fs}{\LS{2}\big(0,T;\HS{1}(\Omega)\big)}^2 +
    \frac{\Cs}{2\epsilon_3}  \hh^2\NORM{\us}{\LS{2}\big(0,T;\HS{1}(\Omega)\big)}^2.
    \label{eq:R3:bound}
  \end{align}

  \PGRAPH{Estimate of $\TERM{R}{4}$}
  To derive an upper bound for $\TERM{R}{4}$, we introduce a piecewise
  linear polynomial approximation $\uss{n+1}_{\pi}$ to $\uss{n+1}$,
  use linear polynomial consistency
  property~\eqref{eq:ash:consistency},
  inequality~\eqref{eq:projection}, and the Young inequality with the
  real coefficient $\epsilon_4>0$:
  \begin{align*}
    &\ABS{ \as(\uss{n+1},\thes{n+1}) - \ash(\uss {n},\thes{n+1}) }
    =    \ABS{ \as(\uss{n+1}-\uss{n+1}_{\pi},\thes{n+1}) - \ash(\uss{n}-\uss{n+1}_{\pi},\thes{n+1}) }
    \\[0.5em]
    &\quad\leq (1+\alpha^*)\NORM{\uss{n+1}-\uss{n+1}_{\pi}}{1,\hh}\,\NORM{\thes{n+1}}{1}
    \leq (1+\alpha^*)
    \bigg(
    \frac{1}{2\epsilon_4}\NORM{\uss{n+1}-\uss{n+1}_{\pi}}{1,\hh}^2+
    2\epsilon_4\NORM{\thes{n+1}}{1}^2
    \bigg)
    \\[0.5em]
    &\quad\leq (1+\alpha^*)\bigg(
    \frac{\Cs}{2\epsilon_4}\hh^2\SNORM{\uss{n+1}}{2}^2
    +2\epsilon_4\SNORM{\thes{n+1}}{1}^2
    \bigg).
  \end{align*}
  Using inequality~\eqref{eq:theta:upper:bound:H1} yields the desired
  upper bound for $\TERM{R}{4}$:
  \begin{align}
    \ABS{\TERM{R}{4}}
    &\leq \Delta\ts\sum_{n=0}^{\Ns-1}\ABS{ \as(\uss{n+1},\thes{n+1}) - \ash(\uss {n},\thes{n+1}) }
    \nonumber\\[0.5em]
    &\leq
    \frac{\Cs(1+\alpha^*)}{\epsilon_4}\hh^2\NORM{\us}{ \LS{\infty}\big(0,T;\HS{2}(\Omega)\big)}^2 +
    4(1+\alpha^*)\,\epsilon_4\Delta\ts\sum_{n=0}^{\Ns-1}\as(\ess{n+1},\ess{n+1}).
    \label{eq:R4:bound}
  \end{align}
  
  \PGRAPH{Estimate of $\TERM{R}{5}$}
  To derive an upper bound for $\TERM{R}{5}$, we introduce a piecewise
  linear polynomial approximation $\partial\uss{n}_{\pi}$ to
  $\partial\uss{n}$, and use the
  relation~\eqref{eq:global:consistency:discrepancy:def}
  \begin{align*}
    \scal{\partial\uss{n}}{\thes{n+1}} - \msh(\partial\uss{n},\thes{n+1})
    =
    \scal{\partial\uss{n}-\partial\uss{n}_{\pi}}{\thes{n+1}} - \msh(\partial\uss{n}-\partial\uss{n}_{\pi},\thes{n+1})
    + \calMh(\thes{n+1},\partial\uss{n}_{\pi})
  \end{align*}
  where the bilinear form $\calMh(\cdot,\cdot)$ is the consistency
  discrepancy
  that stems out of
  using the projector $\PiLS{}$ in the definition of
  $\msh(\cdot,\cdot)$.
  Using inequality~\eqref{eq:projection}, to estimate for the
  approximation error $\partial(\uss{n+1}-\uss{n+1}_{\pi})$, and the
  Young inequality with the real coefficient $\epsilon_5>0$:
  \begin{align*}
    &\ABS{ \scal{\partial(\uss{n}-\uss{n}_{\pi})}{\thes{n+1}} - \msh(\partial(\uss{n}-\uss{n}_{\pi}),\thes{n+1}) }
    \leq (1+\mu^*)\NORM{\partial(\uss{n}-\uss{n}_{\pi})}{0}\,\NORM{\thes{n+1}}{0}
    \\[0.5em]
    &\quad
    \leq (1+\mu^*)
    \bigg(
    \frac{1}{2\epsilon_5}\NORM{\partial(\uss{n}-\uss{n}_{\pi})}{0}^2+
    2\epsilon_5\NORM{\thes{n+1}}{0}^2
    \bigg)
    \leq (1+\mu^*)\bigg(
    \frac{\Cs}{2\epsilon_5}\hh\SNORM{\partial\uss{n}}{1}^2
    +2\epsilon_5\NORM{\thes{n+1}}{0}^2
    \bigg).
  \end{align*}

  We estimate the consistency discrepancy at the time instant
  $\ts^{n+1}$ by applying the result of
  Lemma~\ref{lemma:msh:approx:consistency:global} and apply the Young inequality
  with the real coefficient $\epsilon_6$
  \begin{align*}
    \calMh(\thes{n+1},\partial\uss{n}_{\pi})
    \leq \Cs\hh^2\NORM{\partial\uss{n}_{\pi}}{0}\SNORM{\thes{n+1}}{2}
    \leq \Cs\hh^2\left(
    \frac{1}{2\epsilon_6}\NORM{\partial\uss{n}_{\pi}}{0}^2+
    2\epsilon_6\SNORM{\thes{n+1}}{2}^2
    \right).
  \end{align*}
  The continuity of the projection operator $(\,\cdot\,)_{\pi}$
  implies that
  $\NORM{\partial\uss{n}_{\pi}}{0}\leq\NORM{\partial\uss{n}}{0}$.
  Then, we recall that
  $\partial\uss{n}=(\uss{n+1}-\uss{n})\slash{\Delta\ts}$ and apply the
  Jensen inequality to obtain:
  \begin{align*}
    \NORM{\partial\uss{n}_{\pi}}{0}^2
    \leq \NORM{\partial\uss{n}}{0}^2
    \leq \NORM{\frac{\uss{n+1}-\uss{n}}{\Delta\ts}}{0}^2
    \leq \frac{1}{\Delta\ts^2}\NORM{\int_{\ts^{n}}^{\ts^{n+1}}\frac{\partial\us}{\partial\ts}\dt}{0}^2
    \leq \frac{1}{\Delta\ts}\int_{\ts^{n}}^{\ts^{n+1}}\NORM{ \frac{\partial\us}{\partial\ts} }{0}^2\dt
  \end{align*}
  A straightforward calculation yields
  \begin{align*}
    \Delta\ts\,\sum_{n=0}^{\Ns-1}\frac{1}{\Delta\ts}\int_{\ts^{n}}^{\ts^{n+1}}\NORM{\frac{\partial\us}{\partial\ts}}{0}^2\dt
    = \NORM{ \frac{\partial\us}{\partial\ts} }{ \LS{2}\big(0,\Ts;\LS{2}(\Omega)\big) }^2.
  \end{align*}
  Moreover, we note that
  $\SNORM{\thes{n+1}}{2}\leq\Cs\NORM{\uss{n+1}}{2}$ and a
  straightforward calculation yields
  \begin{align*}
    \Delta\ts\,\sum_{n=0}^{\Ns-1}\SNORM{\thes{n+1}}{2}^2
    \leq \Ts\NORM{\us}{\LS{\infty}(0,\Ts;\HS{2}(\Omega))}^2.
  \end{align*}
  Using these inequalities we find that
  \begin{align}
    \Delta\ts\sum_{n=0}^{\Ns-1}\calMh(\thes{n+1},\partial\uss{n}_{\pi})
    \leq \Cs(1+\mu^*)\,\hh^2\left(
    \frac{1}{2\epsilon_6}
    \NORM{ \frac{\partial\us}{\partial\ts} }{ \LS{2}(0,\Ts;\LS{2}(\Omega)) }^2 + 
    2\epsilon_6\,\Ts\NORM{\us}{\LS{\infty}(0,\Ts;\HS{2}(\Omega))}^2
    \right).
    \label{eq:proof:R5:00}
  \end{align}
  Using inequalities~\eqref{eq:theta:upper:bound:L2}
  and~\eqref{eq:proof:R5:00}, and the result of
  Lemma~\ref{lemma:L2:estimate:u}, cf.
  inequality~\eqref{eq:L2:estimate:u}, the upper bound for
  $\TERM{R}{5}$ finally becomes:
  \begin{align}
    \ABS{\TERM{R}{5}}
    &\leq \Delta\ts\sum_{n=0}^{\Ns-1}\ABS{ \scal{\uss{n+1}}{\thes{n+1}} - \msh(\uss {n},\thes{n+1}) }
    \leq
    4(1+\mu^*)\,\epsilon_5\Delta\ts\sum_{n=0}^{\Ns-1}\as(\ess{n+1},\ess{n+1})
    \nonumber\\[0.5em]
    &\phantom{\leq}
    +\Cs(1+\mu^*)\left( \frac{1}{2\epsilon_5} +  \frac{1}{2\epsilon_6} \right)
    \hh^2
    \bigg( 
    \NORM{ \frac{\partial\us}{\partial\ts} }{ \LS{2}\big(0,\Ts;\LS{2}(\Omega)\big) }^2 +
    (1+\Ts)\NORM{\us}{\LS{\infty}\big(0,\Ts;\HS{2}(\Omega)\big)}^2
    \bigg).
    \label{eq:R5:bound}
  \end{align}

  \medskip
  We note again that
  $\NORM{\ess{\Ns}}{0}^2\leq\max_{1\leq\ns\leq\Ns}\NORM{\ess{\ns}}{0}^2$
  in $\TERM{S}{1}$.
  Also, we note that the five terms $\TERM{R}{j}$, $j=1,\ldots,5$,
  provide an upper bound of the right-hand side
  of~\eqref{eq:proof:main} with the following structure
  \begin{align}
    \sum_{j=1}^{5}\ABS{\TERM{R}{j}}
    &\leq
    \Css_1\max_{1\leq\ns\leq\Ns}\NORM{\ess{n}}{0}^2 +
    \Css_2\Delta\ts\sum_{n=0}^{\Ns-1}\as(\ess{n+1},\ess{n+1}) +
    \Css_1\NORM{\ess{0}}{0}^2
    \nonumber\\
    &\phantom{\leq}
    +\Css_3\hh^2
    \left(
    \NORM{\us}{\LS{\infty}\big(0,T;\HS{2}(\Omega)\big)}^2 +
    \NORM{\frac{\partial\us}{\partial\ts}}{\LS{2}\big(0,T;\HS{1}(\Omega)\big)}^2 +
    \NORM{\fs}{\LS{\infty}\big(0,T;\LS{2}(\Omega)\big)}^2
    \right),
    \label{eq:bound:R1-R5}
  \end{align}
  where we set the constants $\Css_j$, $j=1,\ldots,4$, as
  \begin{align*}
    \Css_1 &= 2(1+\mu^*)\epsilon_1,\qquad
    \Css_2  = 2(1+\mu^*)(2\epsilon_1+2\epsilon_5)+2(1+\alpha^*)(2\epsilon_2+2\epsilon_4),\\[0.5em]
    \Css_3 &= (1+\mu^*)\bigg( \frac{1}{2\epsilon_1} + \frac{1}{2\epsilon_3} + \frac{1}{2\epsilon_5} + \frac{1}{2\epsilon_6} \bigg),\qquad 
    \Css_4  = (1+\alpha^*)\Big( \frac{1}{2\epsilon_2} + \frac{1}{2\epsilon_4} \Big).
  \end{align*}
  We note that the terms $\TERM{R}{j}$ does not involve any additional
  error contribution in time.
  This remarkable fact is consistent with the fact that the virtual
  element method affects only the space discretization.

  Inequality~\eqref{eq:convergence:theorem} and the theorem assertion
  follow by substituting the error bounds~\eqref{eq:bound:S1-S4}
  and~\eqref{eq:bound:R1-R5} in~\eqref{eq:proof:main}, choosing a
  suitable value for the Young coefficients $\epsilon$, $\epsilon_j$,
  $j=1,\ldots,6$, and \RED{taking the square root of both sides of the
    resulting inequality}.
  Finally, we note that all constants from the bounds of terms
  $\TERM{S}{j}$ and $\TERM{R}{j}$ are independent of $\hh$ and
  $\Delta\ts$, and a unique constant can be set, which is taken into
  account $\widetilde{\Cs}$ in~\eqref{eq:proof:main}, and is
  proportional to $\max(\mu^*,\alpha^*)\slash{\min(\mu_*,\alpha_*)}$.
\end{proof}


\section{Numerical Experiments} 
\label{sec:numerical:experiments}

\begin{figure}[t]
  \centering
  \begin{tabular}{ccc}
    \includegraphics[width=0.25\textwidth]{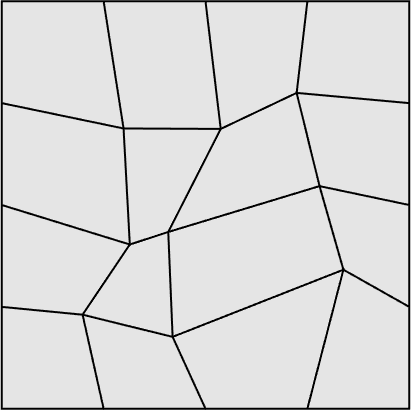} &\quad
    \includegraphics[width=0.25\textwidth]{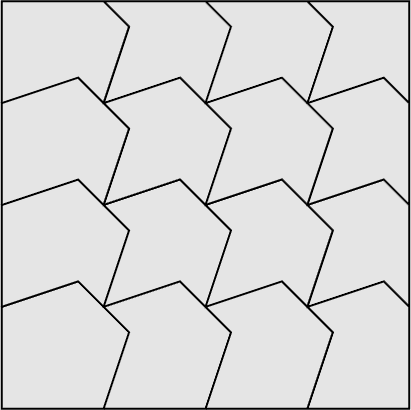}  &\quad
    \includegraphics[width=0.25\textwidth]{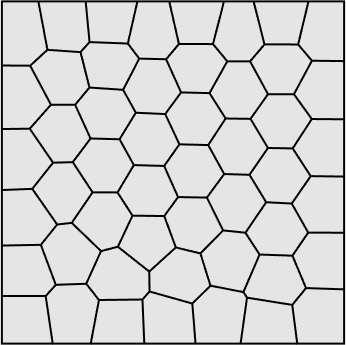} \\
    $(a)$ & $(b)$ & $(c)$
  \end{tabular}
  \caption{a representative mesh of the three mesh families considered
    in the test case: $(a)$ distorted squares; $(b)$ nonconvex
    polygons; $(c)$ Voronoi tesselation.}
  \label{fig:meshes}
\end{figure}

In this section, we apply the virtual element method developed in the
previous sections to the solution of an oscillating circle in two
dimensions.
For the numerical \BLUE{analysis}, domain $\Omega$ is discretized by three
different mesh families, respectively composed by distorted squares,
nonconvex polygons, and smoothed Voronoi tesselations.
Figure \ref{fig:meshes} shows a representative mesh for each family.
The distorted squares and non-convex meshes are based on in-house code
developed in Matlab~\cite{MATLAB:2020}.
The two-dimensional polygonal meshes are generated using the built-in
Matlab function \texttt{voronoin} and the functions in the modules
PolyTop~\cite{Talischi-Paulino-Pereira-Menezes:2012} and
PolyMesher~\cite{Talischi-Paulino-Pereira-Menezes:2012b}.

Let $\essh{n} = \Ussh{n} - \Ih\us(\ts^n)$ be the error in the
approximation of the interpolation of $\us(\ts^n)$ by the virtual
element solution (which is the quantity $\thes{n}$ that is used in the
proof of Theorem~\ref{theorem:convergence}).
We measure the relative approximation error according to this
definition:
\begin{align*}
  \calE
  = \underset{\ns\in[0,\Ns]}{\max} \calE_{0}^{n}
  + \bigg( \Delta\ts\sum_{\ns=0}^{\Ns}         \abs{\calE_{1}^{n}}^2 \bigg)^{\frac12}.
\end{align*}
where 
\begin{align*}
  \calE_{0}^{n} = \left( \frac{ \msh\big(\essh{n},\essh{n}\big) }{ \msh\big(\Ih\us(\ts^n),\Ih\us(\ts^n)\big) }\right)^{\frac12}
  \quad\textrm{and}\qquad
  \calE_{1}^{n} = \left( \frac{ \ash\big(\essh{n},\essh{n}\big) }{ \ash\big(\Ih\us(\ts^n),\Ih\us(\ts^n)\big) }\right)^{\frac12}.
\end{align*}
Before presenting the numerical results, we highlight the
implementation process of the new operator which is defined in
\eqref{eq:PiLSP:vector:def}.
The major difference with the standard VEM is that we use $\PiLSP{}$
instead of $\PizP{}$ in the discretization of the time-derivative term
and the right-hand side term.
We recall that the former is the orthogonal projector onto the
subspace of linear polynomials in every polygonal element with respect
to the discrete inner product~\eqref{eq:discrete:inner:product} and
the latter is the $\LS{2}$ projection operator defined
in~\eqref{eq:orthogonal:projector}.
Our implementation of the local mass matrix on each element $\P$
\BLUE{proceeds in} three steps.
First, we compute the projection matrix
\begin{align}
  \matPiLSP{} = \big(\matD^T\matD\big)^{-1}\matD^T,
  \label{eq:LS:computation}
\end{align}
where we recall that $\matD$ is the matrix collecting the degrees of
freedom of the monomial basis~\eqref{eq:scaled:monomials} and is
defined in~\eqref{eq:matD:def}.
Second, we define the elemental mass matrix $\matM$ using
\eqref{eq:LS:computation}:
\begin{align*}
  \matM = \big(\matPiLSP{}\big)^T\matH\,\matPiLSP{}
  \quad\textrm{where}\qquad
  \restrict{\matH}{ij} = \int_{\P}\ms_i(\xs,\ys)\ms_j(\xs,\ys)\dx\dy
  \quad i,j=1,2,3.
\end{align*}
Third, we assemble the global mass matrix as in the standard finite
element method.
The right-hand side term is also computed using the projection
operator $\PiLSP{}$ in every mesh element.
According to~\eqref{eq:RHS:local:def}, we consider
\begin{align*}
  \bvh = 
  \big(\matPiLSP{}\big)^T\fvh
  \quad\textrm{where}\qquad
  \restrict{\fvh}{i} = \int_{\P}\ms_i(\xs,\ys)\fs(\xs,\ys)\dx\dy.
\end{align*}
The implementation of the stiffness matrix from the bilinear form
$\ash(\cdot,\cdot)$ is carried out as usual in the VEM.
The primary advantage of using the projection operator $\PiLSP{}$ is
that this operator is computable on the original virtual element space
\cite{BeiraodaVeiga-Brezzi-Cangiani-Manzini-Marini-Russo:2013},
whereas we would need the modified virtual element space
\cite{Ahmad-Alsaedi-Brezzi-Marini-Russo:2013} to compute the regular
$\LS{2}$ projection operator $\PizP{}$.

For the numerical computations, we consider the computational domain
$\Omega=[-1,1]^2$ and the time interval, $[0,T] = [0,1/2]$.
We define the noncontact subdomain $\Omega^{+}(t)$ and the contact set
$\Omega^{0}(t)$ as:
\begin{align*}
  \Omega^{+}(\ts) = \big\{ (\xs,\ys)\in\Omega:\rs(\ts)>\rs_0(\ts) \big\}
  \quad\textrm{and}\qquad
  \Omega^{0}(\ts) = \big\{ (\xs,\ys)\in\Omega:\rs(\ts)\leq\rs_0(\ts) \big\},
\end{align*}
where $\rs(\ts)$ and $\rs_0(\ts)$ are respectively given by:
\begin{align*}
  \rs  (\ts) &= \Big( \big(\xs-1/3\cos(4\pi\ts)\big)^2 + \big(\ys-1/3\sin(4\pi\ts)\big)^2\Big)^{\frac12},\\[0.5em]
  \rs_0(\ts) &= 1/3 + 0.3\sin(4\pi\ts).
\end{align*}
The exact solution $u(x,y,t)$ is given by:
\begin{equation*}
  \us(\xs,\ys,\ts) =
  \begin{cases}
    & \dfrac{1}{2}~\Big( \rs^2(\ts)-\rs_0^2(\ts)\Big)^2 \quad           \text{if} \ (\xs,\ys)\in\Omega^{+}(\ts),\\[0.5em]
    & 0                                                \hspace{3.015cm}\text{if} \ (\xs,\ys)\in\Omega^{0}(\ts).
  \end{cases}
  \label{eqn:movcirc_exaana}
\end{equation*}

The initial and boundary conditions can be computed from the exact
solution $\us(\xs,\ys,\ts)$, see~\eqref{eqn:movcirc_exaana}.
The force function is given by:
\begin{equation*}
  \fs(\xs,\ys,\ts) =
  \begin{cases}
    \phantom{-}
    4 \Big[ r_0^2(t)-2r^2(t)-1/2(r^2(t)-r_0^2(t)) \Big( p(t)+r_0(t)~r_0^{'}(t) \Big] & \quad\text{if} \ (\xs,\ys) \in \Omega^{+}(\ts),\\
    -4 r_0^2(t) \Big[ 1-r^2(t)+r_0^2(t)\Big]                                        & \quad\text{if} \ (\xs,\ys) \in \Omega^{0}(\ts),
  \end{cases}
\end{equation*}
Here, $\ps(\ts)$ is defined as
\begin{equation*}
  \ps(\ts) = (x-c_1(t)) c_1'(t) + (y-c_2(t)) c_2'(t),
\end{equation*}
where,
\begin{equation*}
  c_1(t) = \dfrac{1}{3} \cos(4\pi\ts)
  \quad\textrm{and}\quad
  c_2(t) = \dfrac{1}{3} \sin(4\pi\ts)
\end{equation*}
are the centers of the free boundary, which is an oscillatory circle
with radius $r_0(t)$.
It is assumed that the circle is moving with respect to a reference
circle of radius $r_1$ at the origin.
The computations are performed over the three mesh families shown
in Figure \ref{fig:meshes}.
\begin{figure}[t]
  \begin{tabular}{l}
    \includegraphics[width=0.5\textwidth]{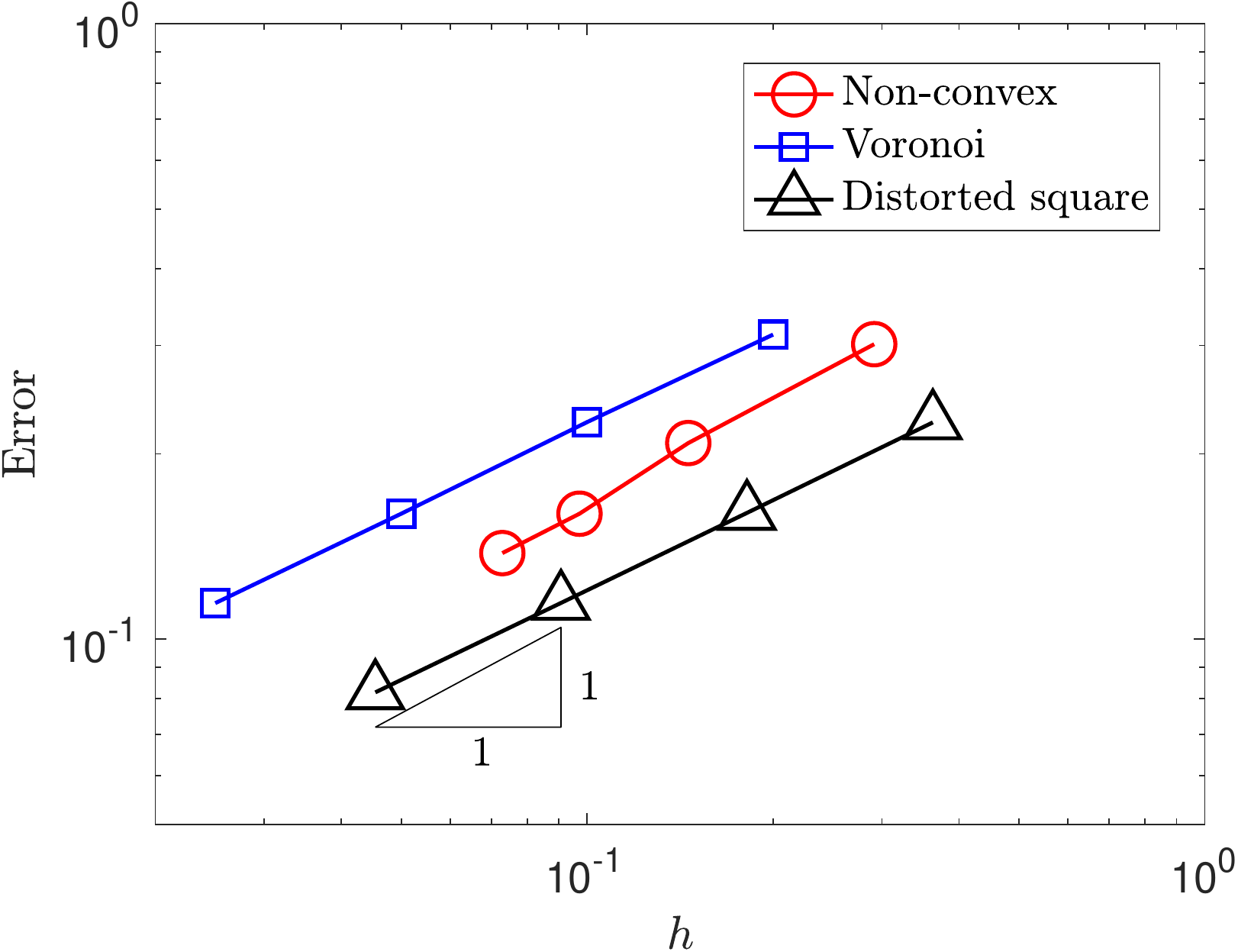}\label{fig:spacecon} 
    \includegraphics[width=0.49\textwidth]{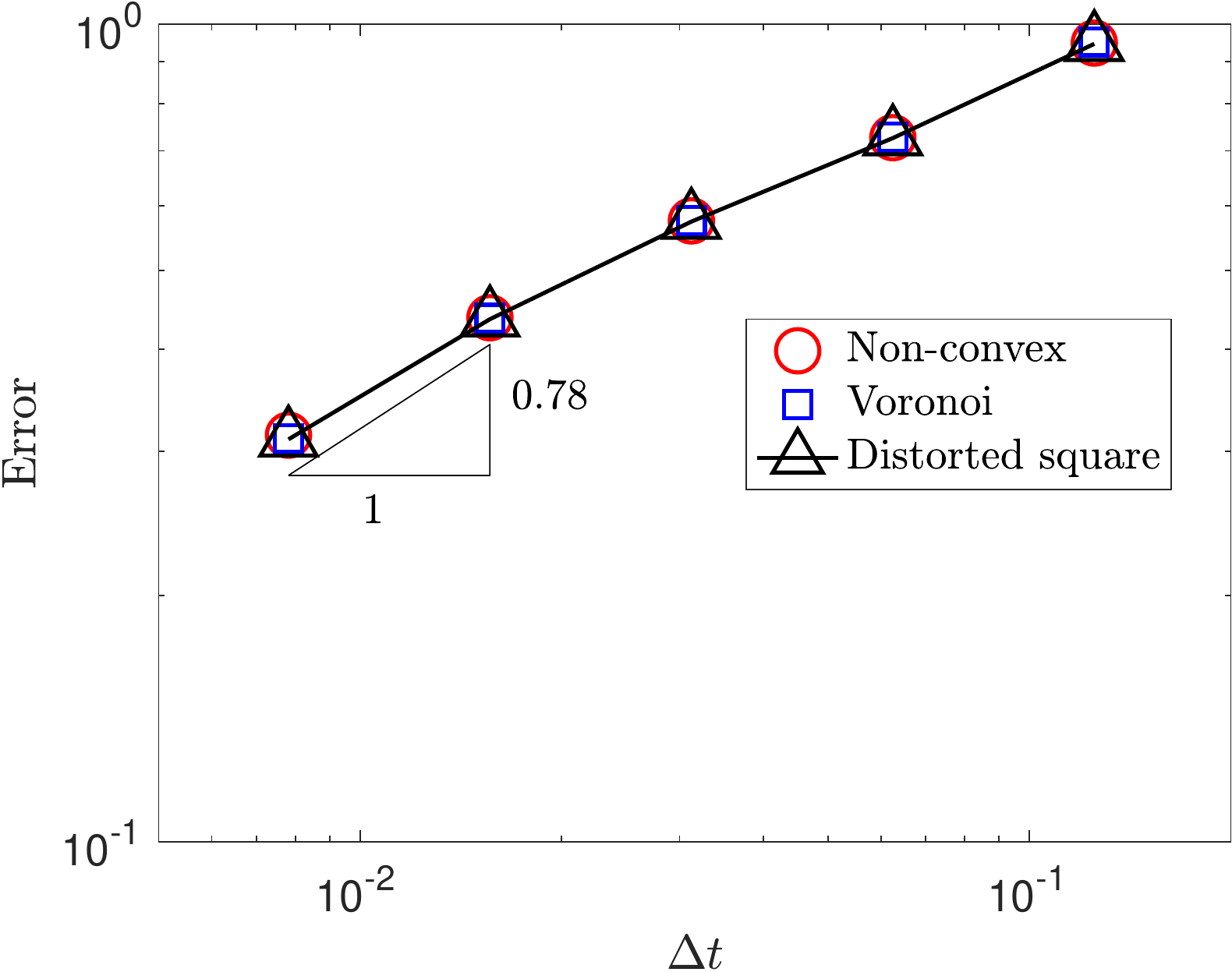}\label{fig:timecon}
  \end{tabular}
  \caption{Convergence of the error \RED{(see
      equation~\eqref{eq:convergence:theorem})} with respect to the
    mesh refinement for a constant $\Delta\ts=10^{-3}$ (top panel) and
    halving $\Delta\ts=$ on a mesh with fixed $\hh$ (bottom panel).}
  \label{fig:convplots}
\end{figure}
\begin{figure}[htpb!]
  \centering
  \includegraphics[width=\textwidth]{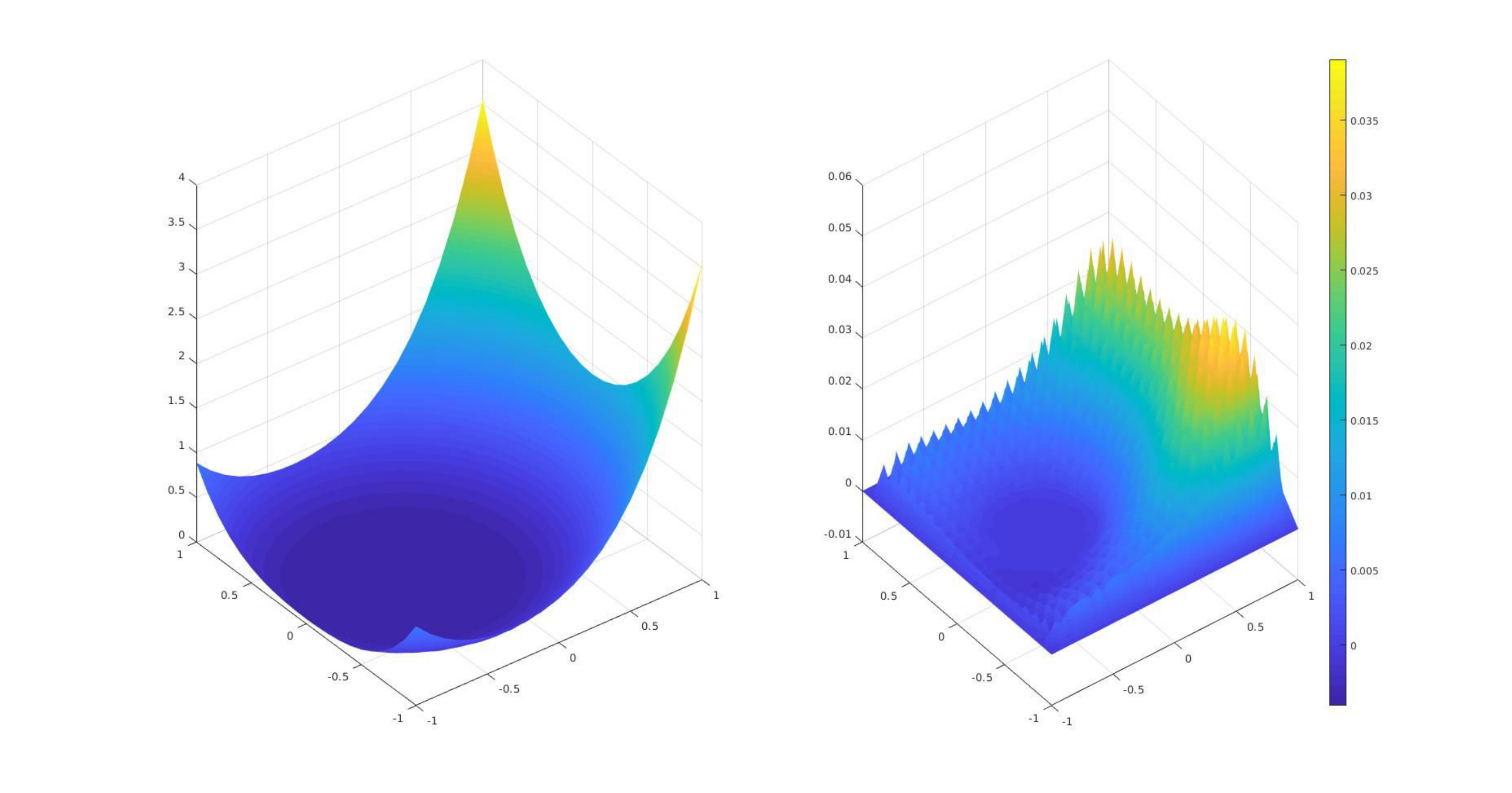}
  \caption{Numerical solution (left panel) and relative approximation
    error (right panel) $u-u_h$ at time $T=.25$.}
  \label{fig:soln:error}
\end{figure}
To study the convergence in space, we consider the time increment
$\Delta\ts=10^{-3}$ and a mesh sequence with initial mesh size as:
$(a)$ distorted mesh, $\hh\approx 0.36$; $(b)$ nonconvex mesh,
$\hh\approx 0.30$; $(c)$ Voronoi tesselation, $\BLUE{\hh=0.20}$.
At each mesh refinement we halve $\hh$.
To study the convergence in time, we halve the time step $\Delta\ts$
at each time refinement starting with $\BLUE{\Delta t=0.125}$ and carry out all
calculation on the following meshes: $(a)$ distorted square mesh with
$\hh=0.045$; $(b)$ nonconvex mesh with $\hh=0.073$; $(c)$ Voronoi mesh
with $\BLUE{\hh=0.025}$.
We choose these mesh sizes in order that the total number of degrees of
freedom on the various meshes is almost the same.
The convergence of the error with mesh size and time increment is
shown in Figures \ref{fig:convplots} for the three mesh families
considered in this test.
The triangles close to the error curves show the numerically computed
rate of convergence.
It can be inferred that the error decreases at the optimal convergence
rate in both the space and temporal variable with order $1$ and
$\approx0.75$, respectively, in agreement
with~Theorem~\ref{theorem:convergence}.
\BLUE{ Finally, Figure~\ref{fig:soln:error} shoes the numerical
  solution at (left panel) and the corresponding distribution in space
  of the relative approximation error (right) the final time $T=0.25$.
}


\section{Conclusions}
\label{sec:conclusions}

We designed, analyzed and numerically tested a virtual element method
for solving the parabolic variational inequality problem in two
dimensions over unstructured polygonal meshes.
Several aspects make this design challenging.
In particular, we used the Maximum and Minimum Principle Theorem to
ensure that a virtual element function is nonnegative if all its
degrees of freedom are nonnegative.
We introduced an approximate orthogonal projector onto linear
polynomials, whose approximation properties are carefully investigated
in this paper, to compute the mass matrix.
The convergence analysis requires a nonnegative quasi-interpolation
operator, whose construction on polygonal elements is also discussed
in the paper.
We proved a convergence theorem and estimated that the convergence
rate \textcolor{blue}{is} proportional to $\hh$ (the mesh size parameter) and
$\Delta\ts^{\frac34}$ (the time step parameter).
These results are in perfect agreement with a previous finite element
formulation from the literature working in triangular meshes \textcolor{blue}{\cite{Johnson:1976}}.
We assessed the behavior of the VEM against a manufactured solution
problem on a two-dimensional domain defined by an oscillating circle
using three different polygonal mesh families including distorted
squares, nonconvex elements, and Voronoi tesselations.
All the numerical convergence rates reflected by the slope of the
error curves in our log-log plots agree with the rates that are
expected from the theory.


\section*{Acknowledgments}
\RED{
We acknowledge the anonymous Reviewers for their invaluable comments
and, in particular, for suggesting us an alternative proof of the
existence of the positive quasi-interpolant, which we added to the
paper in a final appendix.
}
GM has been partially supported by the ERC Project CHANGE, which has
received funding from the European Research Council (ERC) under the
European Unions Horizon 2020 research and innovation programme (grant
agreement No 694515).
Dibyendu Adak was partially supported by the Institute Postdoctoral
fellowship at Department of Mechanical Engineering, Indian Institute
of Technology-Madras.



\appendix
\section*{\RED{Appendix: An alternative construction of the quasi-interpolation operator}}

\RED{
In this section, we would like to outline an alternative proof of
Lemma~4.5.
For each polygonal element $\P$, we define the quasi-interpolant
operator $\Ih\restrict{\us}{\P}\in\HS{1}(\P)$ as the solution of the
following Poisson's equation with Dirichlet boundary condition.
\begin{align*}
  \Delta\Ih\us &= 0\phantom{\us_{C}} \quad \text{in~}\P, \\
  \Ih\us       &= \us_{C}\phantom{0} \quad \text{on~}\partial\P,
\end{align*}
where $\us_{C}$ is the \textit{Cl\'ement} interpolation operator on
the sub-triangulation of the mesh $\Th$ by joining each vertex of
$\partial\P$ with with the barycentre of $\P$.
Following \cite{Mora-Rivera-Rodriguez:2015}, it can be proved that
$\Ih\us$ approximates $\us$ optimally, i.e.
\begin{align*}
  \norm{\us-\Ih\us}{0,\P} + \hP\norm{\us-\Ih\us}{1,\P} \leq
  \Cs\hP^2\snorm{\us}{2,\P}.
\end{align*}
Furthermore, from the construction, on each node $\nu$,
$\Ih\us(\nu)=\us_C(\nu)$, and
\begin{align*}
  \us_C(\nu) = \frac{1}{|\omega_{\nu}|}\int_{\omega_{\nu}}\us\,d\xv,
\end{align*}
where $\omega_{\nu}$ is the patch of the node $\nu$ on the
sub-triangulation of $\Th$.
Consequently, we have $\Ih\us(\nu)\geq0$, when $\us\geq0$ almost
everywhere on $\Omega$.
Moreover, $\Ih\us$ is a harmonic function.
Therefore, from corollary~\eqref{lemma:VEM:nonnegative-subset:global},
we emphasize that $\Ih\us\in\mathcal{K}_{\hh}$ if $\us\in\mathcal{K}$.
}

\end{document}